\documentclass[11pt]{article}
\usepackage{amsmath}
\usepackage{amsthm}
\usepackage{amssymb}
\usepackage{amsfonts}
\usepackage{latexsym}
\usepackage{graphicx}
\usepackage{caption}
\usepackage[pdftex,bookmarks=true,bookmarksnumbered=true]{hyperref}
\usepackage{color}
\renewcommand{\epsilon}{\varepsilon}
\renewcommand{\rho}{\varrho}
\renewcommand{\phi}{\varphi}
\newcommand{\DS}{\displaystyle}
\newcommand{\N}{{\mathbb N}}
\newcommand{\Z}{{\mathbb Z}}

\newcommand{\R}{{\mathbb R}}

\newcommand{\cF}{{\cal F}}

\newcommand{\cL}{{\cal L}}

\newcommand{\cX}{{\cal X}}
\newcommand{\cY}{{\cal Y}}
\begin{document}
\newtheorem{definition}{Definition}[section]
\newtheorem{theorem}[definition]{Theorem}
\newtheorem{proposition}[definition]{Proposition}
\newtheorem{corollary}[definition]{Corollary}
\newtheorem{lemma}[definition]{Lemma}
\newtheorem{hypothesis}[definition]{Hypothesis}
\newtheorem{remark}[definition]{Remark}
\newtheorem{conjecture}[definition]{Conjecture}
\newtheorem{problem}[definition]{Problem}
\title{Cyclic Symmetry Induced Pitchfork Bifurcations \\
       in the Diblock Copolymer Model}
\author{{\em Peter Rizzi, Evelyn Sander, and Thomas Wanner} \\[2ex]
    	 Department of Mathematical Sciences, George Mason University, \\
           Fairfax, VA 22030, USA  
    }
%
\date{May 18, 2023}
\maketitle
%
%
\begin{abstract} 
The Ohta-Kawasaki model for diblock copolymers exhibits a rich
equilibrium bifurcation structure. Even on one-dimensional base domains
the bifurcation set is characterized by high levels of
multi-stability and numerous secondary bifurcation points. Many of these
bifurcations are of pitchfork type. In previous work, the authors
showed that if pitchfork bifurcations are induced by a simple~$\Z_2$
symmetry-breaking, then computer-assisted proof techniques can be used
to rigorously validate them using extended systems. 
However, many diblock copolymer pitchfork bifurcations cannot be treated
in this way. In the present paper, we show that in these more involved
cases, a cyclic group action is responsible for their existence, based
on cyclic groups of even order. We present theoretical results
establishing such bifurcation points and show that they can be
characterized as nondegenerate solutions of a suitable extended
nonlinear system. Using the latter characterization, we also demonstrate that  
computer-assisted proof techniques can be used to validate such bifurcations. While the
methods proposed in this paper are only applied to the diblock copolymer
model, we expect that they will also apply to other parabolic partial differential
equations.


\end{abstract}
\newpage
\setcounter{tocdepth}{2}
\tableofcontents
%
%
%
\section{Introduction}
\label{sec:intro}
Symmetry-breaking pitchfork bifurcations are a common feature of 
nonlinear partial differential equation models as they vary with respect to parameters. 
We focus here on pitchfork bifurcations of the one-dimensional Ohta--Kawasaki
model for the formation of diblock copolymers~\cite{ohta:kawasaki:86a}. 
In a previous paper~\cite{lessard:sander:wanner:17a}, we developed a rigorous
computer-assisted proof method for the validation of 
symmetry-breaking pitchfork bifurcations in the case of $\Z_2$-symmetries that
were observed in ~\cite{johnson:etal:13a}. 
These results were based on creating a validated version of the numerical methods of Werner 
and Spence~\cite{werner:spence:84a}, by reformulating the existence of the bifurcation point
as the existence of a nondegenerate solution of an extended nonlinear system.

However, there are cases of pitchfork bifurcations that  we observed in~\cite{johnson:etal:13a,
lessard:sander:wanner:17a}, but were unable to validate using the above theoretical methods --- as although a
high degree of symmetry is broken at bifurcation, all local solutions remain in the same
$\Z_2$-symmetry class, i.e., solutions are all even or all odd with respect to the center of the
one-dimensional domain. In this paper, we adapt the techniques used for the $\Z_2$-symmetry case in order to give 
theoretical underpinnings needed for 
rigorous computational  validation of symmetry-breaking bifurcations for more  general cyclic group symmetries, 
as shown in Figures~\ref{fig:oddodd} and~\ref{fig:eveneven} and described in more detail below. 
\begin{figure}
\begin{center}
\includegraphics[width=0.24\textwidth]{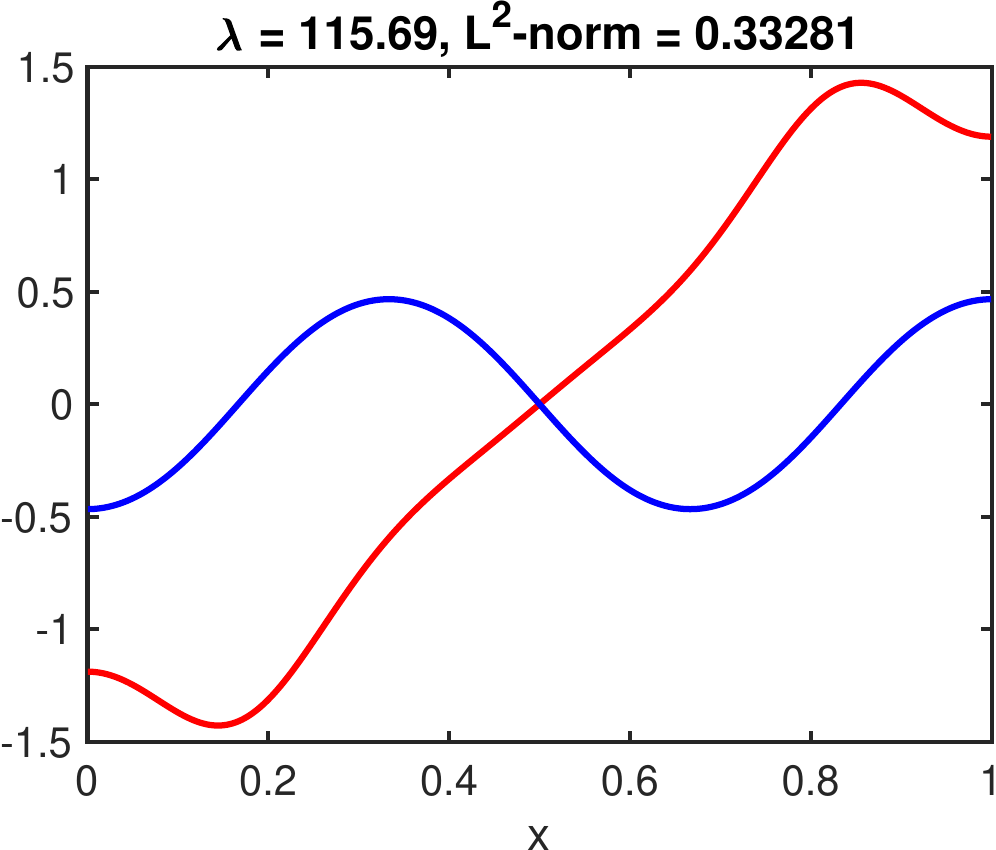}
\includegraphics[width=0.24\textwidth]{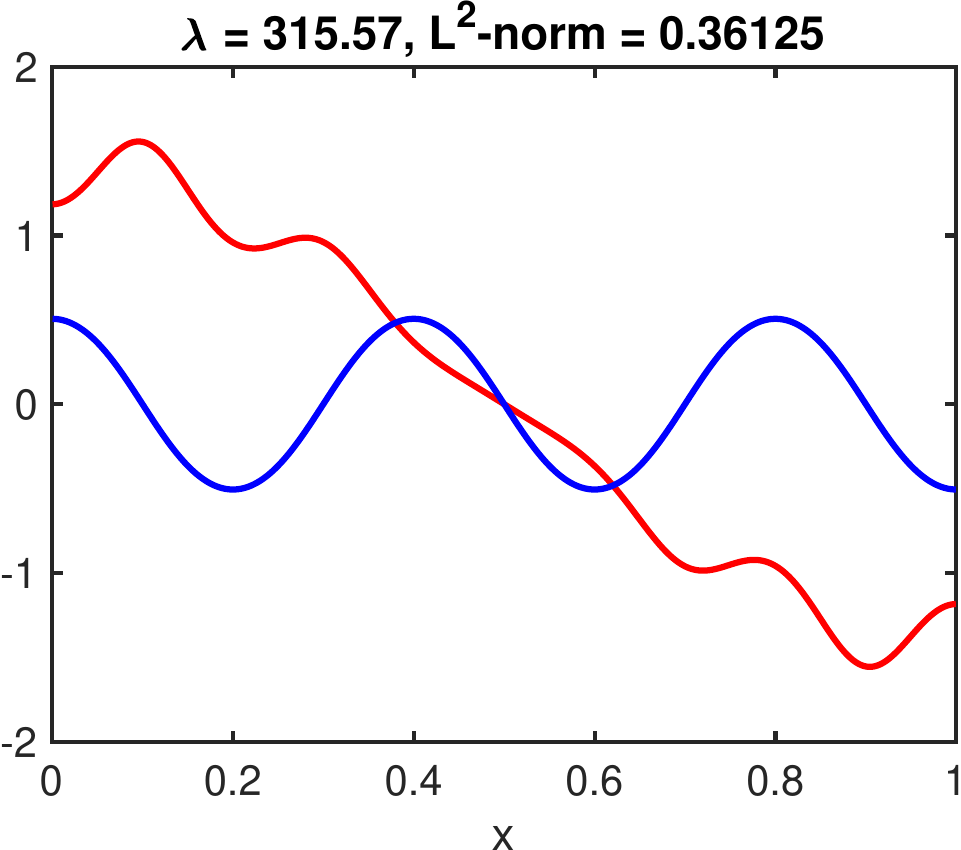}
\includegraphics[width=0.24\textwidth]{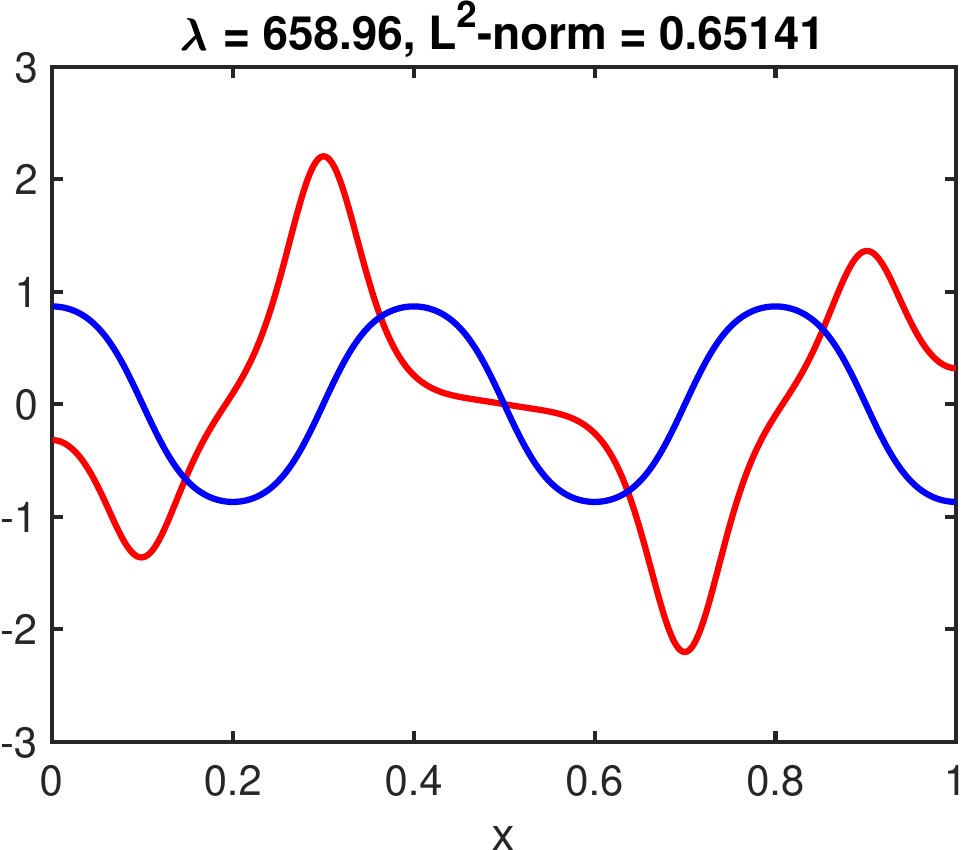}
\includegraphics[width=0.24\textwidth]{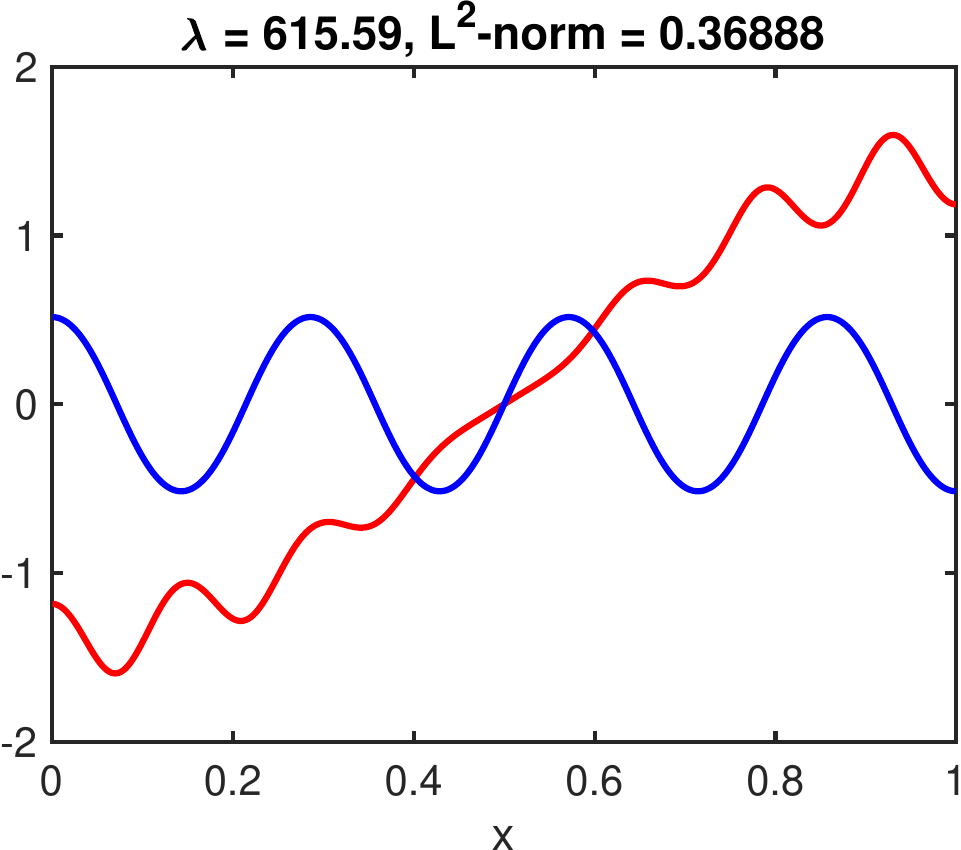}\\
\includegraphics[width=0.25\textwidth]{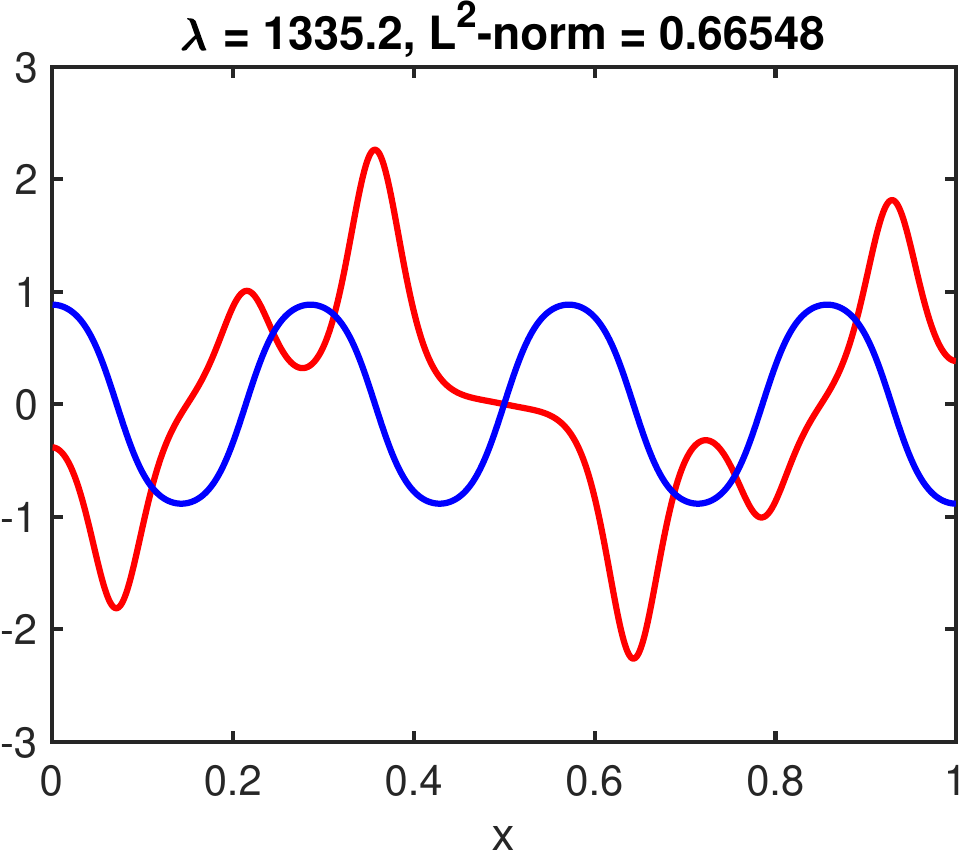}
\includegraphics[width=0.25\textwidth]{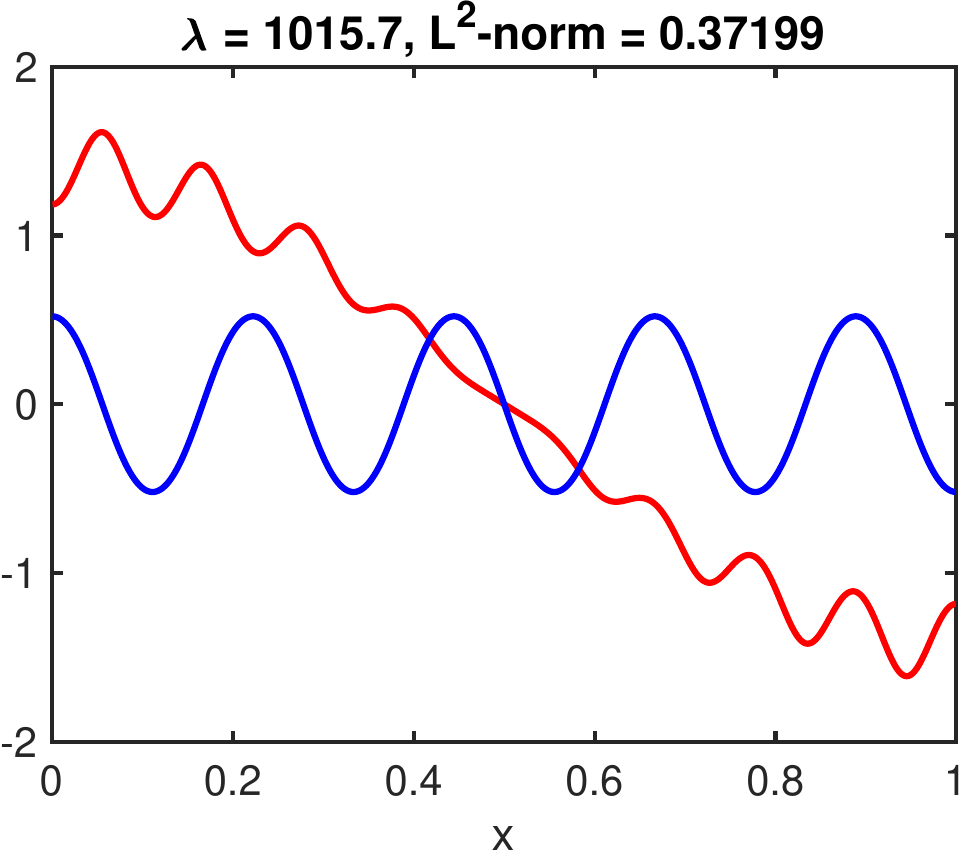}
\includegraphics[width=0.25\textwidth]{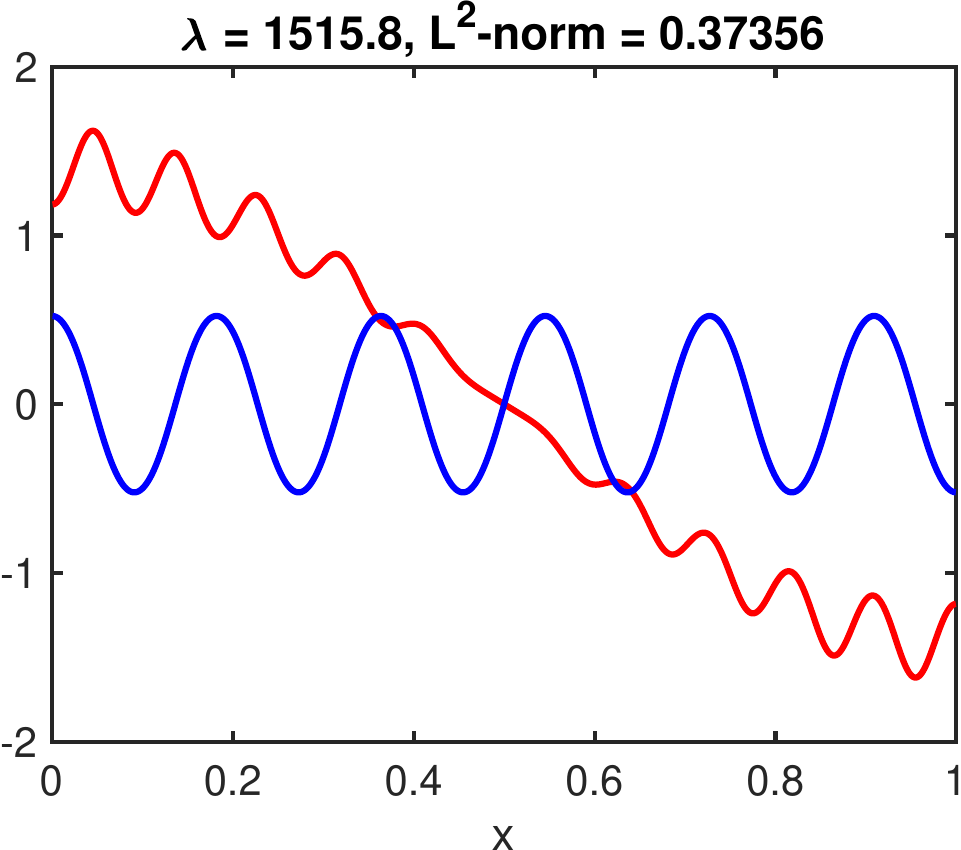}
\caption{\label{fig:oddodd} Examples of odd solutions, shown in blue, with odd
eigenfunctions, depicted in red, which correspond to eigenvalue zero at the
bifurcation point, where odd is measured with respect to the domain midpoint~$1/2$.
The solutions $u$ are $n$-layer solutions, equivariant under the cyclic
symmetry~(\ref{eqn:z2nsymm}), where $n$ is~$3$, $5$, $5$, $7$ in the top
row, and~$7$, $9$, $11$ in the bottom row, respectively. Since both the
bifurcating solution and the eigenfunction are odd, the solutions remain
odd as they undergo a pitchfork bifurcation, but bifurcation breaks the
cyclic symmetry. 
}
\end{center}
\end{figure}
\begin{figure}
\begin{center}
\includegraphics[width=0.24\textwidth]{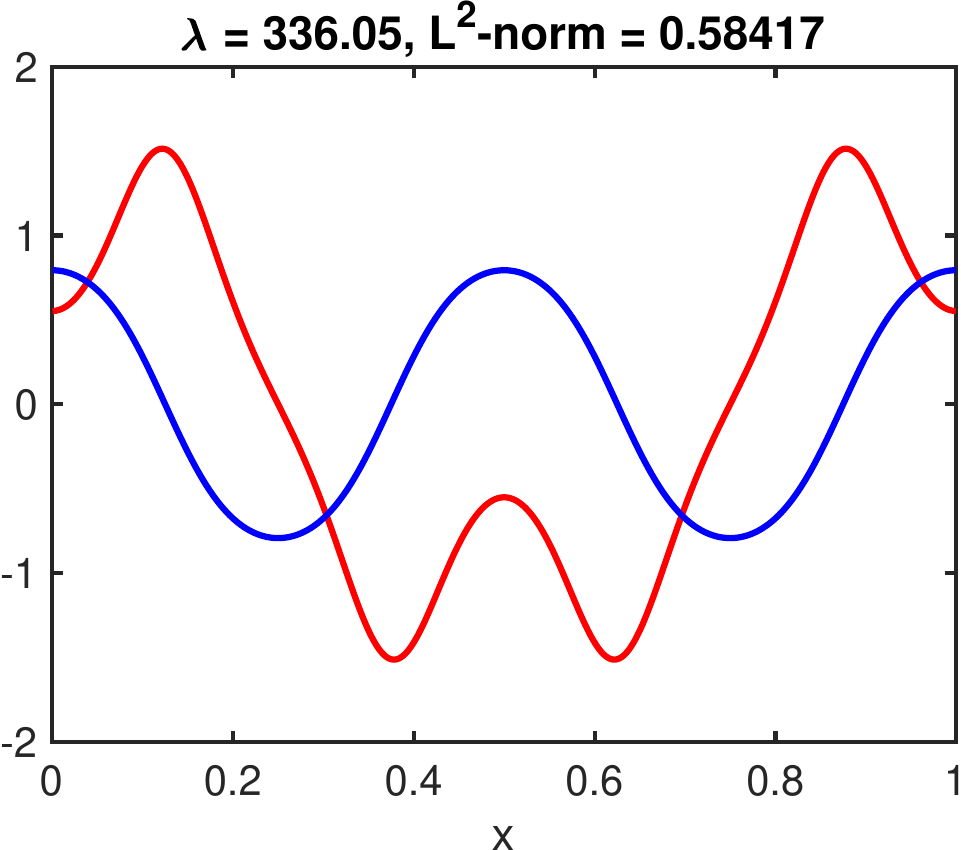}
\includegraphics[width=0.24\textwidth]{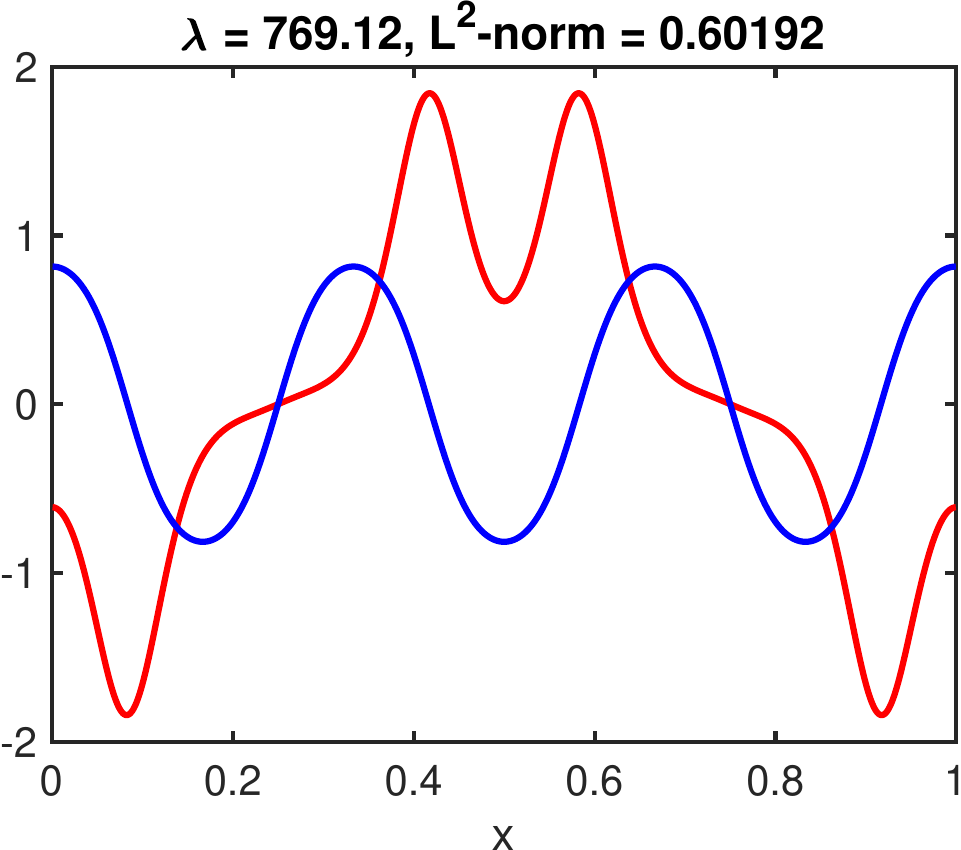}
\includegraphics[width=0.24\textwidth]{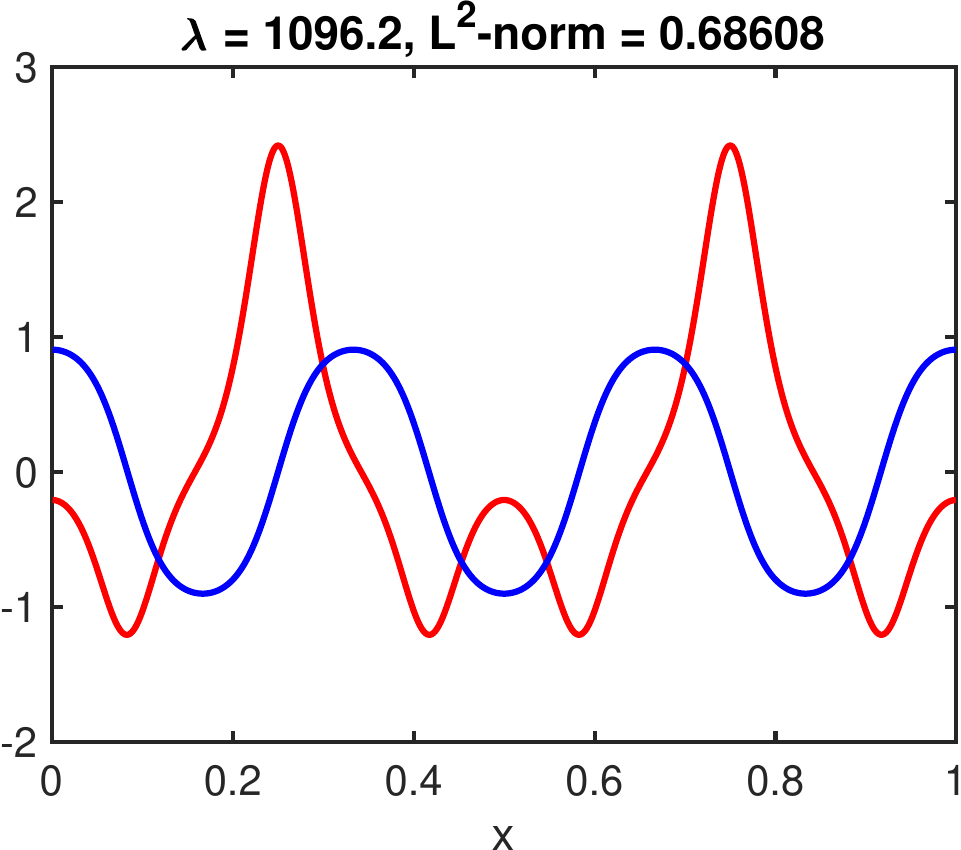}
\includegraphics[width=0.24\textwidth]{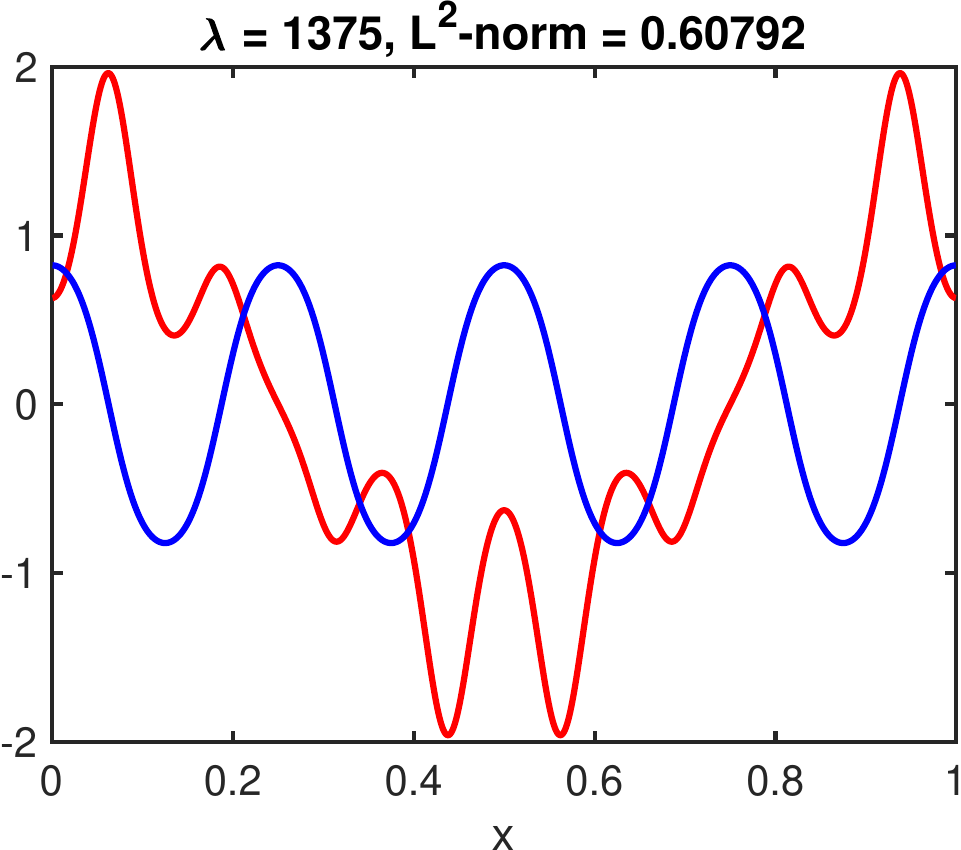}
\caption{\label{fig:eveneven} Examples of even solutions, shown in blue,
with even eigenfunctions, depicted in red, which correspond to eigenvalue
zero at the bifurcation point, where even is measured with respect to the
domain midpoint~$1/2$. In each figure, the solutions~$u$ are $n$-layer
solutions, equivariant under the cyclic symmetry~(\ref{eqn:z2nsymm}),
where $n$ is~$4$, $6$, $6$, and~$8$, respectively. Since both the
bifurcating solution and the eigenfunction are even, solutions remain
even as they undergo a pitchfork bifurcation, but the bifurcation
breaks the cyclic symmetry. 
}
\end{center}
\end{figure}

The Ohta-Kawasaki equation is a model for diblock copolymers, materials formed by 
two linear polymers (known as blocks) which contain different monomers. If the blocks
are thermodynamically incompatible, then the blocks try to separate after the reaction.
However, since they are covalently bonded, such a separation is impossible on the
macroscopic scale. This competition of long range and short range forces causes
microphase separation, resulting in pattern formation. The Ohta-Kawasaki equation 
on a domain $\Omega \subset \R^d$ is  given by  
\begin{eqnarray*}
  w_t &=& -\Delta ( \Delta w + \lambda f(w) ) - \lambda \sigma (w - \mu)
  \quad\mbox{ in } \Omega \;, \\[1.5ex]
  & & \frac{ \partial w }{\partial \nu } =
    \frac{\partial (\Delta w)}{\partial \nu }= 0
    \quad\mbox{ on } \partial \Omega \; ,
\end{eqnarray*}
where~$\nu$ denotes the unit outward normal on the boundary
of~$\Omega$, corresponding to homogeneous Neumann boundary conditions.
The quantity $w(t,x)$ is the local density difference of the two monomer blocks.
That is, if $w(t,x) = -1$, then at time~$t$ and locally near the point~$x$, the 
material consists entirely of block~A. If instead we have $w(t,x) = 1$,
then the local average of the material consists entirely of block~B. For
values $-1<w(t,x)<1$ the local material contains a mix of blocks A and B. 
The parameter $\mu$ is the space average of~$w$, meaning it is a measure
of the relative total proportion of the two polymers, which we tersely
refer to as the {\em mass} of the system.  The equation obeys a mass
conservation,  implying that~$\mu$ is time-invariant. A large value of
the parameter~$\lambda$  corresponds to a large short-range repulsion,
while a large value of the parameter~$\sigma$ corresponds to large
long-range elasticity forces. We refer the reader
to~\cite{johnson:etal:13a} for a detailed description of
how~$\lambda$ and~$\sigma$ are defined. See also~\cite{wanner:16a} for a
description of the phase separation aspects of the model. Finally, note
that the second boundary condition is necessary since this is a fourth
order equation. In this paper, we focus on equilibrium solutions $w = w(x)$. 

For notational convenience, we reformulate our equation slightly. 
For a solution~$w$ of the diblock copolymer equation, we define
$u = w - \mu$. Since the space average of~$w$ is~$\mu$, the average
of the shifted function~$u$ is zero. Therefore the equilibrium equation
becomes
\begin{eqnarray} \label{eqn:dbcp}
-\Delta ( \Delta u + \lambda f(u+\mu) ) - \lambda \sigma u  &=&
  0 \quad\mbox{ in } \Omega \;, \nonumber\\[1.5ex]
\frac{ \partial u }{\partial \nu }=  \frac{ \partial (\Delta u) }
  {\partial \nu }&=& 0 \quad\mbox{ on } \partial \Omega \;,  \\[1.5ex]
\int_\Omega u \;dx &=& 0 \nonumber \;. 
\end{eqnarray}
We will consider this version of the equation  for the rest of the paper,
and restrict our attention to the case of the one-dimensional domain
$\Omega = (0,1)$ with $\mu = 0$, where $\sigma > 0$ denotes a fixed constant,
and the nonlinearity is chosen as $f(u) = u-u^3$. Note that while this
particular form of the nonlinearity is not critical for our results, the
fact that the nonlinearity is odd plays a large role for the results 
of this paper. We would like to point out, however, that this oddness
condition was chosen purely to simplify our presentation. One could
in fact obtain similar results without it.
\begin{figure}
\begin{center}
\includegraphics[width=0.49\textwidth]{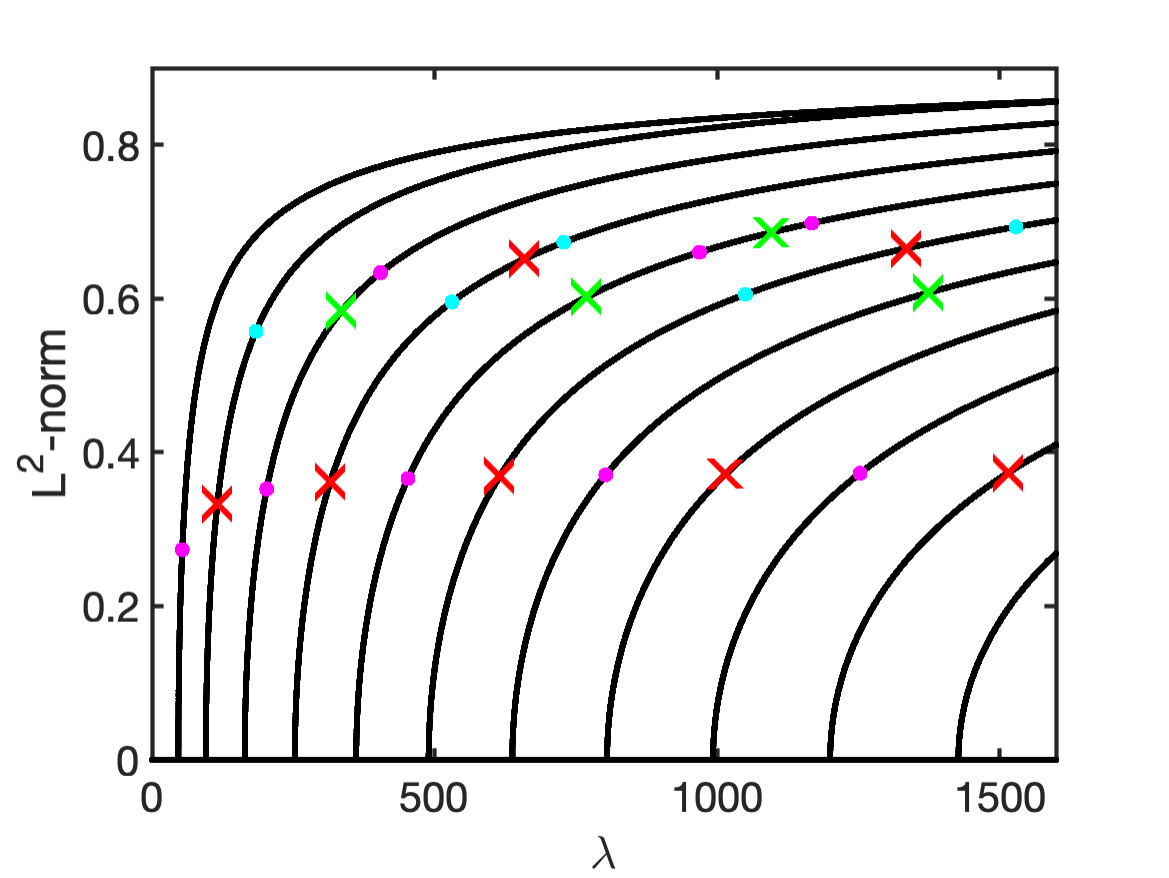}
\includegraphics[width=0.49\textwidth]{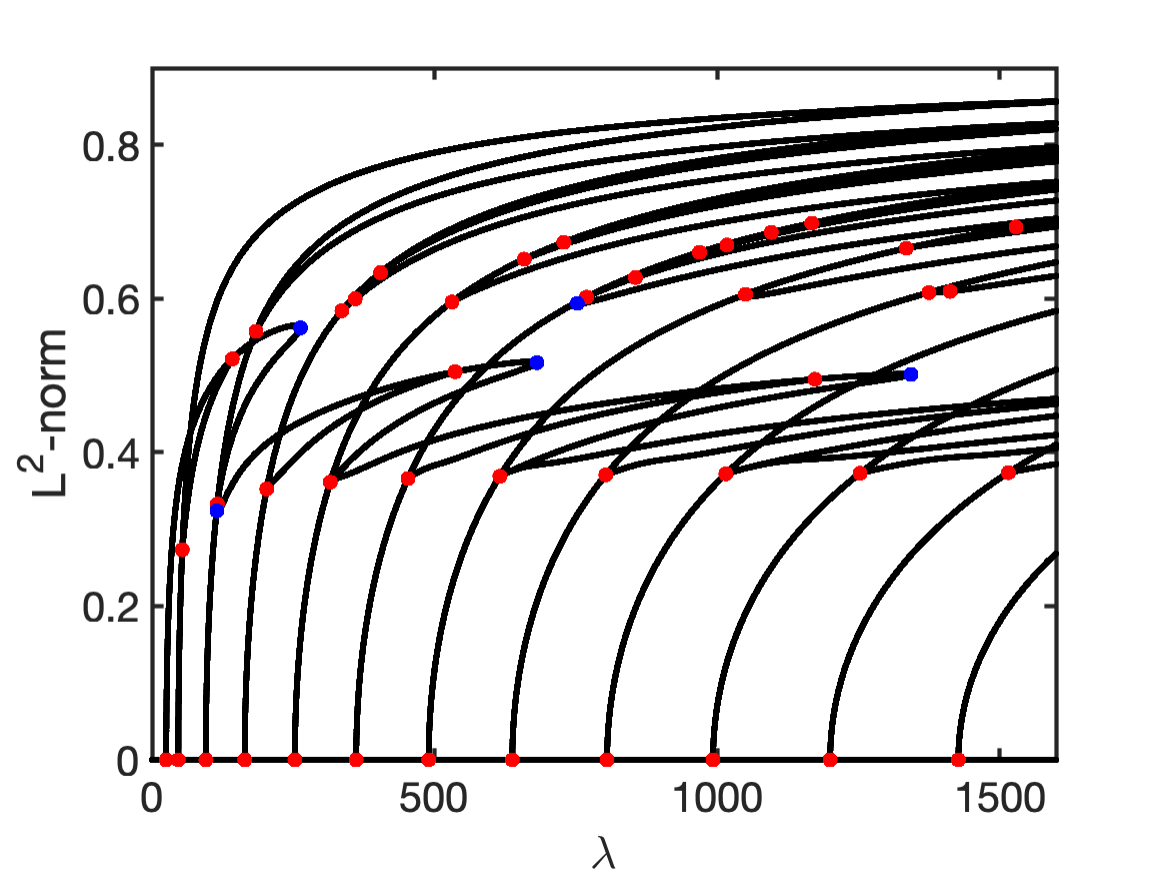}
\caption{\label{fig:bifdiag} Partial bifurcation diagram for the
one-dimensional diblock copolymer equation for the parameter values
$\mu = 0$ and $\sigma = 6$. In the left panel, each of the dots
and crosses is a bifurcation point. The red crosses correspond to
the odd solutions with odd eigenfunctions shown in Figure~\ref{fig:oddodd},
while the green crosses correspond to the even solutions with even
eigenfunctions shown in Figure~\ref{fig:eveneven}. In contrast, the
cyan dots represent odd solutions with even eigenfunctions, while the
magenta dots correspond to even solutions with odd eigenfunctions.
These last two types are $\Z_2$-symmetry breaking bifurcations points.
Altogether, the colored points depict all detected bifurcation points along
primary branches. Originating at each bifurcation point, there are
secondary branches which are omitted for the sake of clarity. They
are included, however, in the right panel, which illustrates that
the branches are connected through multiple routes.}
\end{center}
\end{figure}

Figure~\ref{fig:bifdiag} shows a numerically computed bifurcation diagram
for~(\ref{eqn:dbcp}) with $\sigma=6$. The bifurcation diagram is restricted to
the primary equilibrium branches emanating from the spatially homogeneous trivial
solution, along with the secondary bifurcation points shown as red dots. Secondary
branches do emanate from each of these branches, but they have been omitted for the
sake of clarity. Some of the depicted bifurcation points are $\Z_2$-symmetry breaking,
as covered in~\cite{lessard:sander:wanner:17a}. However, in Figures~\ref{fig:oddodd}
and~\ref{fig:eveneven} we show that for nine cases, there is no $\Z_2$-symmetry
broken at the pitchfork bifurcation. Instead, these bifurcation solutions are
$n$-layer solutions which are equivariant under the following cyclic symmetry.

Suppose that~$u$ denotes the solution at one of these bifurcation points, and
let~$\phi$ denote the eigenfunction of the Fr\'echet derivative of~(\ref{eqn:dbcp})
at~$u$ corresponding to the eigenvalue~$0$, which we further assume to be simple.
Furthermore, suppose that we have extended the solution~$u$ from~$\Omega = (0,1)$
to all of~$\R$ via successive even reflections across the boundary. Then each
such solution~$u$ satisfies the cyclic symmetry given by
\begin{equation} \label{eqn:z2nsymm}
  \left( T_n u \right)(x) \; = \; 
  -u\left( x + \frac{1}{n} \right) \; = \; u(x)
  \qquad\mbox{ for all }\qquad
  x \in \R \; ,
\end{equation}
for some $n \in \N$. In contrast, the eigenfunction~$\phi$ does not display this
type of symmetry, in fact, it seems to have no special symmetry properties at all.
We will see later that the operator~$T_n$ is the generator of a cyclic group, but
since its natural functional-analytic domain interferes with our homogeneous Neumann 
boundary conditions, we defer precise statements about the spaces on which~$T_n$
is defined and the order of the generated cyclic group until the next section.

In this paper, we develop the theoretical foundation for a rigorous computer-assisted
proof method for showing that the functions shown in Figures~\ref{fig:oddodd}
and~\ref{fig:eveneven} do indeed give rise to symmetry-breaking pitchfork bifurcations.
This is accomplished by first establishing a mathematical bifurcation result on pitchfork
bifurcations induced by a cyclic group action, and then equivalently reformulating it
as a zero-finding problem for a suitable extended nonlinear system, in the spirit
of~\cite{lessard:sander:wanner:17a}. The latter system can then in principle be solved
using computer-assisted proofs based on the constructive implicit function theorem
introduced in~\cite{sander:wanner:16a, wanner:17a}. Although in this paper we
concentrate on the Ohta-Kawasaki equation, our methods can be adapted to a much
more general set of equations. In particular, our pitchfork bifurcation result is
quite general, in that along with technical assumptions, it only relies on being
able to divide the space into the direct sum of pairwise orthogonal spaces which
exhibit certain invariance properties. One of these spaces contains
the function at bifurcation, and another contains the eigenfunction spanning the associated
one-dimensional kernel of the operator. The result does not include the specific
form~(\ref{eqn:z2nsymm}) of the symmetry in the statement, but in our application,
these pairwise orthogonal spaces come from symmetry considerations. Note that
while the solution-eigenfunction pairs shown in Figures~\ref{fig:oddodd}
and~\ref{fig:eveneven} appear to be orthogonal in the sense that
$\int_0^1 u(x) \phi(x) \, dx = 0$, there is no immediately obvious
symmetry which would force this identity. In fact, it will be shown later that
despite the lack of any obvious symmetry, each eigenfunction~$\phi$ exhibits
a more subtle one, which is far from obvious and which is responsible for this
orthogonality.

A number of papers have previously considered numerical computation of bifurcation
diagrams for the Ohta-Kawasaki and Cahn-Hilliard equations, such as for
example~\cite{choksi:etal:11a, desi:etal:11a, johnson:etal:13a, maier:etal:08a}.
There are also several decades of results on computer validation
for dynamical systems and differential equations solutions which combine fixed
point arguments and interval arithmetic, see for example~\cite{arioli:koch:10a,
gameiro:etal:08a, nakao:etal:19a, plum:95a, rump:10a, yamamoto:98a}. There are
several papers that have already considered rigorous validation of parameter-dependent
solutions for the Ohta-Kawasaki model~\cite{cai:watanabe:19a, sander:wanner:21a,
vandenberg:williams:17a}, as well as in other contexts~\cite{cyranka:wanner:18a,
kamimoto:kim:sander:wanner:22a, maier:etal:07a}. However, this is the first study
to look at computer-assisted proofs of higher symmetry-breaking bifurcations for the
Ohta-Kawasaki model. 

The remainder of this paper is organized as follows. In Section~\ref{sec:symmetry}
we describe the symmetry spaces associated with the cyclic group action given
by~(\ref{eqn:z2nsymm}), discuss their essential properties, and explain why the
underlying symmetry group responsible for the bifurcation is in fact~$\Z_{2n}$. In
Section~\ref{sec:pitbif}, we state and prove the analytical $\Z_{2n}$-equivariant
pitchfork bifurcation result. Furthermore, by reformulating the problem as a
zero-finding problem for an extended system, we also are able to establish computationally
testable conditions. We would like to point out that while the results of this section
are formulated only for the specific situation considered in this paper, the general
approach should be applicable in much more generality. This is described in more 
detail in the context of Remark~\ref{remark:metathm}, which collects the essential
assumptions that are necessary. In Section~\ref{sec:compval}, we introduce the computational
validation methods required, based on recent results from~\cite{rizzi:etal:22a}. 
This paper is primarily focused on the analysis of this new type of symmetry-breaking 
pitchfork bifurcations, but for proof of concept,  we end the paper with sample solution
validations of pitchfork bifurcation points from Figures~\ref{fig:oddodd}
and~\ref{fig:eveneven}. Nevertheless, a few computational challenges remain,
and we briefly address these as well as potential solution attempts.
%
%
\section{Cyclic Equivariance of Diblock Copolymers}
\label{sec:symmetry}
In this section we describe the equivariance properties of the equilibrium
diblock copolymer model with respect to the cyclic symmetry mentioned in
the introduction. We start in Section~\ref{subsec:symmetry1} by presenting
our basic functional-analytic setup. After briefly discussing the difficulties 
with using the cyclic symmetry defined in~(\ref{eqn:z2nsymm}) directly in
this framework, Section~\ref{subsec:symmetry2} shows that these issues can
be overcome by considering a detour through larger function spaces. In combination
with additional symmetries, one can then in fact use the symmetry operator~$T_n$
to derive a suitable symmetry-induced decomposition of the original spaces.
We also present orthogonality and invariance properties which are essential
for our application. Finally, Section~\ref{subsec:symmetry3} establishes
necessary equivariance properties of the nonlinear diblock copolymer
operator.
%
%
\subsection{Basic functional-analytic setup}
\label{subsec:symmetry1}
We begin by briefly presenting the functional-analytic framework for our
study of the diblock copolymer model, which had already been used in our
previous work~\cite{johnson:etal:13a, lessard:sander:wanner:17a,
sander:wanner:21a, wanner:17a}. As mentioned in the introduction, we
restrict consideration to equilibria for the Ohta-Kawasaki model given
in~\eqref{eqn:dbcp} on the one-dimensional domain $\Omega = (0,1)$, and
study only the zero mass case~$\mu=0$. In addition, the parameter~$\sigma$
is fixed and strictly positive, and we use the cubic odd nonlinearity
$f(u) = u-u^3$. From a functional-analytic point of view, we can then
rearrange the system of three equations in such a way that the equilibrium
solutions are zeros of a single nonlinear operator~$F$, while the boundary
and integral conditions are absorbed into the definition of the operator
domain. In this way, equilibria of the problem~\eqref{eqn:dbcp} correspond
to solutions of the zero finding problem
\begin{equation} \label{def:Flambdau}
  F(\lambda,u) =
  -\Delta ( \Delta u + \lambda f(u + \mu) ) - \lambda \sigma u = 0
  \; ,
\end{equation}
where we have $F : \R \times X \to Y$ with respect to the spaces
\begin{equation} \label{def:xyspaces}
  X = \left\{ u \in H^2(\Omega): \frac{\partial u}{\partial \nu} = 0
    \mbox{ on } \partial \Omega, \;\mbox{ and }\;
    \int_\Omega u \, dx = 0 \right\}
  \quad\mbox{ and }\quad
  Y = H^{-2}(\Omega) \; .
\end{equation}
These spaces are both Hilbert spaces, equipped for our purposes with
the norms
\begin{displaymath}
  \| u \|_X = \| \Delta u \|_{L^2(\Omega)}
  \qquad\mbox{ and }\qquad
  \| u \|_Y = \left\| \Delta^{-1} u \right\|_{L^2(\Omega)} \; ,
\end{displaymath}
where one can immediately verify that the mapping $\Delta: L^2(\Omega)
\cap \left\{ \int_\Omega u \, dx = 0 \right\} \to H^{-2}(\Omega)$
is an isometry. Standard results imply that in this setting the
operator~$F$ is a well-defined smooth operator. Furthermore, since
we assumed the identity $\mu = 0$ and~$f$ is an odd function, we also
have $F(\lambda,-u) = -F(\lambda,u)$ for all $\lambda \in \R$ and
$u \in X$.

Of particular importance for the detection of bifurcation points
is of course the Fr{\'e}chet derivative of~$F$ at a given equilibrium
solution. Thus, in the following, we consider a fixed parameter
value $\lambda_0 \in \R$ and a function $u_0 \in X$, and we let~$L$
denote the Fr\'echet derivative of~$F$ at the pair~$(\lambda_0,u_0)$
given by
\begin{equation} \label{eqn:dfu}
  L[v] \; = \;
  D_uF(\lambda_0,u_0)[v] \; = \;
  -\Delta (\Delta v + \lambda f'(u_0+\mu) v) - \lambda \sigma v
  \; .
\end{equation}
According to its definition, one has $L = D_uF(\lambda_0,u_0) \in \cL(X,Y)$,
where~$\cL(X,Y)$ denotes the Banach space of all bounded linear operators
from~$X$ to~$Y$, equipped with the operator norm~$\|\cdot\|_{\cL(X,Y)}$.
With the range and null space of this linear operator we associate the
following orthogonal projections.
\begin{definition}[Orthogonal projections~$P$ and~$Q$ associated with~$L$]
\label{def:orthproj}
Let~$L$ be the Fr\'echet derivative of the diblock copolymer operator
as defined in~\eqref{eqn:dfu}. Then we denote by~$Q : X \to X$ the orthogonal
projection of the domain~$X$ onto the null space~$N(L)$, and we let~$P : Y \to Y$
be the orthogonal projection of~$Y$ onto the orthogonal complement of the
range~$R(L)$. In other words, we define the closed subspaces $\tilde{X} \subset X$
and $\tilde{Y} \subset Y$ such that
\begin{displaymath}
  X = N(L) \oplus \tilde{X}
  \;\;\mbox{ and }\;\;
  Y = \tilde{Y} \oplus R(L) \; ,
  \quad\mbox{ where }\quad
  \tilde{X} \perp N(L)
  \;\;\mbox{ and }\;\;
  \tilde{Y} \perp R(L) \; .
\end{displaymath}
Thus, the projector $P: Y \to Y$ is characterized by $R(P) = \tilde{Y}$
and $N(P) = R(L)$, while the projection $Q: X \to X$ satisfies both
$R(Q) = N(L)$ and $N(Q) = \tilde{X}$.
\end{definition}
These projection operators allow us to apply standard arguments based on
the Lyapunov-Schmidt reduction to establish pitchfork bifurcations induced
by the action of a cyclic group.

Finally, we impose the following two assumptions on the operator~$F$ and the
pair~$(\lambda_0,u_0)$, which are standard for the discussion of bifurcation
points. The first of these has been verified for the diblock copolymer operator
in~\cite{lessard:sander:wanner:17a}, and the second one will be satisfied at
all pitchfork bifurcation points shown in Figures~\ref{fig:oddodd}
and~\ref{fig:eveneven}.
\begin{hypothesis}[Fredholm property]
\label{hyp:fred}
Assume that the nonlinear operator $F : \R \times X \to Y$ is sufficiently
smooth, and suppose that the pair~$(\lambda_0,u_0) \in \R \times X$ is a zero
of the operator~$F$, i.e., we assume that it satisfies the identity
$F(\lambda_0,u_0) = 0$. Furthermore, suppose that the Fr\'echet
derivative~$L = D_uF(\lambda_0,u_0)$ of~$F$ at~$(\lambda_0,u_0)$
is a Fredholm operator of index zero.
\end{hypothesis}
\begin{hypothesis}[One-dimensional kernel]
\label{hyp:1dk}
Suppose that the above Fredholm Hypothesis~\ref{hyp:fred} is satisfied.
In addition, assume that the Fr\'echet derivative~$L$ has a one-dimensional
null space. Since~$L$ is of index zero, this immediately implies that its range
has codimension one. Therefore, there exist nonzero elements $\phi_0 \in X$
and $\psi_0^* \in Y^*$ such that both
\begin{displaymath}
  N(L) = \mathrm{span}(\phi_0)
  \qquad\mbox{ and }\qquad
  R(L) = N(\psi_0^*)
\end{displaymath}
are satisfied. Furthermore, these assumptions show that the projections~$P$
and~$Q$ from Definition~\ref{def:orthproj} both have rank one.
\end{hypothesis}
%
%
\subsection{Space decompositions induced by cyclic symmetry}
\label{subsec:symmetry2}
We now turn our attention to studying the cyclic symmetry operator~$T_n$
defined in~\eqref{eqn:z2nsymm}. As we mentioned in the introduction, all
of the pitchfork bifurcation equilibria shown in Figures~\ref{fig:oddodd}
and~\ref{fig:eveneven} are fixed points of this operator. Note that since
the definition of~$T_n$ includes a shifted argument, we had to extend the
definition of the underlying functions beyond the bounded domain
$\Omega = (0,1)$ by even reflections. More precisely, consider for the
moment an arbitrary function $u \in X$, where the space~$X$ was defined
in~\eqref{def:xyspaces} above. In view of the imposed homogeneous Neumann
boundary conditions, we can extend the function~$u$ smoothly to a periodic
function~$\tilde{u}$ on~$\R$, by first defining
\begin{displaymath}
  \tilde{u}(x) = \left\{  
  \begin{array}{ccc}
    u(x)   & \mbox{ for }  & 0 \le x \le 1 \; , \\[1.5ex]
    u(2-x) & \mbox{ for }  & 1 < x < 2 \; ,
  \end{array} \right.
\end{displaymath}
and then using the identity~$\tilde{u}(x + 2k) = \tilde{u}(x)$
for all $ x \in [0,2]$ and $k \in \Z$. In the following, we will
refer to~$u \in X$ and~$\tilde{u} : \R \to \R$ as {\em corresponding
functions\/}, or equivalently, we will say that~{\em $\tilde{u}$ is
the extension of~$u$\/}.

With the notion of corresponding functions, it now makes sense
to apply the symmetry operator~$T_n$. One can immediately see that for
the equilibrium-eigenfunction pairs shown in Figures~\ref{fig:oddodd}
and~\ref{fig:eveneven}, the extension~$\tilde{u}$ of every solution~$u$
is indeed a fixed point of the operator~$T_n$, where~$n$ denotes the
number of layers of~$u$. On the other hand, extensions~$\tilde{\phi}$
of the eigenfunctions~$\phi$ are not fixed points of~$T_n$.

In view of these observations, it seems plausible to expect that 
standard approaches to studying symmetry-induced bifurcations should
apply directly in our situation, such as the ones described
in~\cite{chossat:lauterbach:00a, golubitsky:etal:88a}. Notice, however,
that for these approaches to work one needs to study symmetry operators
which are {\em acting on the space containing the equilibrium solutions\/}
--- and in our case this is the Hilbert space~$X$. Yet, one can easily see
that if~$\phi$ is one of the eigenfunctions in Figures~\ref{fig:oddodd}
or~\ref{fig:eveneven}, then the function~$T_n(\tilde{\phi})$ no
longer satisfies homogeneous Neumann boundary conditions on~$\Omega$.
In other words, {\em the restriction of~$T_n(\tilde{\phi})$ to~$\Omega$
is no longer an element of~$X$\/}. In addition, even in situations where
one can use abstract equivariant bifurcation theory, one still has to verify
the actual bifurcation type and the corresponding nondegeneracy conditions
in the specific underlying system, as explained in detail
in~\cite{maier:wanner:97a}.

At first glance, this observation appears to doom the use of the
symmetry operator~$T_n$. Nevertheless, we will show in the remainder
of this subsection that this is far from the truth. In fact, we will
be able to study~$T_n$ on a larger Hilbert space which contains~$X$,
but on which the action of~$T_n$ is well-defined --- and then use
the obtained insight to construct an appropriate space decomposition
of~$X$.

To introduce this larger space, we first return to the definition of
the function $\tilde{u} : \R \to \R$ corresponding to an element
$u \in X$. Notice that according to our construction this extension
satisfies $\tilde{u}(x+2) = \tilde{u}(x)$ for all $x \in \R$. Thus,
its restriction to~$(0,2)$ automatically satisfies the periodic
boundary conditions $\tilde{u}(0) = \tilde{u}(2)$ and $\tilde{u}'(0)
= \tilde{u}'(2) = 0$. Furthermore, one can immediately see that the
symmetry operator~$T_n$ defined in~\eqref{eqn:z2nsymm} maps every
$2$-periodic function to another $2$-periodic function.

With this in mind, we introduce three spaces of $2$-periodic functions,
two of which will extend our Hilbert spaces~$X$ and~$Y$. More precisely,
we consider
\begin{eqnarray}
  H^2_{per}(\R) & = & \left\{ v \in H^2_{loc}(\R) \; : \;
    v(x+2) = v(x) \;\mbox{ for all }\; x \in \R \; , \;\;
    \int_0^2 v \, dx = 0  \right\} \; , \nonumber \\[2ex]
  L^2_{per}(\R) & = & \left\{ v \in L^2_{loc}(\R) \; : \;
    v(x+2) = v(x) \;\mbox{ for all }\; x \in \R \; , \;\;
    \int_0^2 v \, dx = 0  \right\} \; , 
    \label{def:wperspaces} \\[2ex]
  H^{-2}_{per}(\R) & = & \Delta L^2_{per}(\R) \; ,
    \nonumber
\end{eqnarray}
where the space~$L^2_{loc}(\R)$ denotes the space of all measurable
real-valued functions on~$\R$ which are square integrable on compact
intervals, and~$H^2_{loc}(\R) \subset L^2_{loc}(\R)$ the space of all
twice weakly differentiable Sobolev functions whose first two derivatives
are in~$L^2_{loc}(\R)$ as well. All three of the above spaces are 
Hilbert spaces with respect to the norms
\begin{displaymath}
  \begin{array}{rclcc}
    \DS \|v\|_2 & = & \DS \| \Delta v\|_{L^2(0,2)} &
      \mbox{ for } & \DS v \in H^2_{per}(\R) \; , \\[1.5ex]
    \DS \|v\|_0 & = & \DS \| v\|_{L^2(0,2)} &
      \mbox{ for } & \DS v \in L^2_{per}(\R) \; , \\[1.5ex]
    \DS \|v\|_{-2} & = & \DS \| \Delta^{-1} v\|_{L^2(0,2)} &
      \mbox{ for } & \DS v \in H^{-2}_{per}(\R) \; ,
  \end{array}
\end{displaymath}
respectively. Since we have restricted ourselves to functions with
mean zero, one can readily verify that both mappings $\Delta :
H^2_{per}(\R) \to L^2_{per}(\R)$ and $\Delta : L^2_{per}(\R) \to
H^{-2}_{per}(\R)$ are isometries.

Our interest in these spaces is two-fold. On the one hand, they
are spaces of periodic functions which in some sense contain our
fundamental Hilbert spaces~$X$ and~$Y$. To see this, note that
for every function $u \in X$, its extension clearly satisfies
$\tilde{u} \in H^2_{per}(\R)$. However, it is not true in general
that the restriction of every function in the latter space lies
in~$X$. Nevertheless, for any $u \in X$ and $k \in \Z$, the construction of its
corresponding function~$\tilde{u}$ implies
\begin{displaymath}
  \tilde{u}(-x) =
  \tilde{u}( \underbrace{-x+2k}_{\in [0,2]} ) =
  \tilde{u}( 2 - (-x+2k) ) =
  \tilde{u}( x + 2(1-k) ) =
  \tilde{u}(x)
\end{displaymath}
for all $x \in \R$, i.e., the extension is an even function.
Conversely, one can easily see that every even function $v \in
H^2_{per}(\R)$ satisfies $v(1+x) = v(-(1+x)) = v(2 - (1+x)) =
v(1-x)$, i.e., it is also even with respect to $x = 1$. Since
every function in~$H^2_{per}(\R)$ is continuously differentiable
in view of Sobolev's embedding theorem~\cite{adams:fournier:03a},
this in turn implies that the restriction of any even function
$v \in H^2_{per}(\R)$ to~$\Omega$ satisfies homogeneous Neumann
boundary conditions, and we have shown that in fact
\begin{equation} \label{eqn:xsubw}
  X = \left\{ v|_\Omega \; : \;
    v \in H^2_{per}(\R) \;\mbox{ and }\;
    v(x) = v(-x) \;\mbox{ for all }\; x \in \R \right\}
  \; .
\end{equation}
Our second interest in the above spaces stems from the fact that
they are invariant under the symmetry operator defined in~\eqref{eqn:z2nsymm}.
More precisely, let~$W$ denote either~$H^2_{per}(\R)$ or~$L^2_{per}(\R)$.
Then we can clearly define an isometry~$T_n : W \to W$ via
\begin{displaymath}
  T_n(v)(x) = -v\left( x + \frac1n \right) \; .
\end{displaymath}
In addition, for $W = H^{-2}_{per}(\R)$ one can set $T_n(v) = \Delta
T_n(\Delta^{-1}v )$. In other words, the symmetry operator~$T_n$ is
a well-defined action on these spaces.

As we stated above, we will use the operator~$T_n$ on the spaces
of periodic functions to ultimately introduce a decomposition of
the spaces~$X$ and~$Y$ from the last subsection. For this, however,
we need to first study~$T_n$ on the former spaces. In the following,
we begin by considering the cases $W = H^2_{per}(\R)$ and
$W = L^2_{per}(\R)$, since in these cases the elements are actually
functions that can be evaluated pointwise. The Sobolev space with
negative exponent will be treated subsequently.

It is immediately clear that the operator~$T_n : W \to W$ has the
property of being cyclic of order~$2n$, and that it commutes with the
Laplacian, i.e., we have both
\begin{displaymath}
  T_n^{2n} = I
  \qquad\mbox{ and }\qquad
  \Delta T_n = T_n \Delta \; .
\end{displaymath}
The first property in particular implies that the minimal polynomial
for~$T_n$ is given by
\begin{displaymath}
  m(t) \; = \; t^{2n}-1 \; = \; (t^n-1)(t^n+1) \; = \;
  \underbrace{(t-1)}_{m_a(t)}
    \underbrace{(t^{n-1}+t^{n-2} + \dots + 1)}_{m_b(t)}
    \underbrace{(t^n+1)}_{m_c(t)} \; .
\end{displaymath}
In addition, one can verify by direct computation that $(t^n + 1)/2 -
(t^n - 1)/2 = 1$ and that~$1$ is not a root of either~$m_b$ or~$m_c$.
This in turn implies that the three polynomials~$m_a$, $m_b$, and~$m_c$
are relatively prime, and therefore we have the decomposition
\begin{equation} \label{eqn:wabc}
  \begin{array}{ccccccc}
    W & = & \DS N(T_n-I) & \oplus & \DS N(T_n^{n-1} + T_n^{n-2} +
      \dots + I) & \oplus & \DS N(T_n^n + I) \\[1.5ex]
    & = & \DS \underbrace{N(m_a(T_n))}_{W_a} & \oplus & \DS
      \underbrace{N(m_b(T_n))}_{W_b} & \oplus &
      \DS \underbrace{N(m_c(T_n))}_{W_c} \; .
  \end{array}
\end{equation}
Each of these three subspaces has additional important properties which
are crucial for our applications, and which will be studied in more
detail below. For now, we would like to point out that the space~$W_a$
consists of functions~$v$ with $v(x + 1/n) = -v(x)$ for all $x \in \R$.
Thus,  by inspection, one would suspect that for the diblock copolymer equilibrium solutions~$u$
shown in Figures~\ref{fig:oddodd} and~\ref{fig:eveneven}, their respective
corresponding functions~$\tilde{u}$ lie in~$W_a$ for $W = H^2_{per}(\R)$,
if~$n$ denotes the number of layers of~$u$. As we will see later, the
respective eigenfunctions will automatically be contained in one of the
remaining two subspaces.

The decomposition of the space~$W$ into a direct sum of three
subspaces lies at the heart of our approach, and we will now show
that this decomposition can be pulled down to the subspace~$X$.
We have already seen in~\eqref{eqn:xsubw} that the space~$X$ occurs
naturally as a subspace of~$W = H^2_{per}(\R)$ if we additionally
impose an evenness constraint. As the following result shows, this
latter constraint plays well with the decomposition~$W = W_a \oplus
W_b \oplus W_c$.
\begin{lemma}[Invariance under reflection]
\label{lem:symreflect}
Let $W = H^2_{per}(\R)$ or $W = L^2_{per}(\R)$, and suppose that~$u \in W$
is arbitrary. Furthermore, suppose that~$v \in W$ is defined via $v(x) = u(-x)$
for all $x \in \R$. Then for every $\tau \in \{ a,b,c \}$ one has the implication
\begin{displaymath}
  u \in W_\tau
  \qquad\Longrightarrow\qquad
  v \in W_\tau
  \quad\mbox{ and }\quad
  u+v \in W_\tau \; .
\end{displaymath}
\end{lemma}
\begin{proof}
Notice first that we only have to establish the validity of~$v \in W_\tau$
in the above implication. According to its definition, the space~$W_\tau$
is a linear subspace, and therefore the inclusions $u, v \in W_\tau$
immediately imply $u+v \in W_\tau$ as well.

Now let~$u$ and~$v$ be given as in the formulation of the lemma. Then
the periodicity of~$u$ implies $v(x) = u(-x) = u(2-x)$, and this
in turn yields
\begin{equation} \label{lem:symreflect1}
  T^k_n v(x) =
  (-1)^k v \left(x+\frac{k}{n}\right) =
  (-1)^k u \left( 2-x-\frac{k}{n} \right)
\end{equation}
for all $x \in \R$ and $k \in \N_0$. We now distinguish between
the three cases $\tau \in \{ a,b,c \}$.

To begin with, let $u \in W_a$. Then we have~$u(t) = T_n u(t) =
-u(t + 1/n)$, and this readily implies $u(t) = -u(t-1/n)$. If one
now substitutes $t = 2-x$, then~\eqref{lem:symreflect1} gives
\begin{displaymath}
  T_n v(x) = - u\left( 2-x-\frac{1}{n} \right) = u(2-x) =
  u(-x) = v(x) \; ,
\end{displaymath}
i.e., we also have $v \in W_a$.

Consider now the case $u \in W_b$. Then the equation $\sum_{k=0}^{n-1}
T_n^k u(s) = 0$ holds for all $s \in \R$. Therefore, if we set $s =
2 - x - 1 + 1/n$, then one obtains with~\eqref{lem:symreflect1} the
identity
\begin{eqnarray*}
  \sum_{k=0}^{n-1} T_n^k v(x) & = &
    \sum_{k=0}^{n-1} (-1)^k  u\left( 2-x-\frac{k}{n} \right) \; = \;
    \sum_{\ell=0}^{n-1} (-1)^{n-1-\ell} u\left( 2-x-\frac{n-1-\ell}{n}
    \right) \\[1.5ex]
  & = & (-1)^{n-1} \sum_{\ell=0}^{n-1} (-1)^{\ell}
    u\left( s + \frac{\ell}{n} \right) \; = \;
    (-1)^{n-1} \sum_{\ell=0}^{n-1} T_n^\ell u(s) \; = \; 0 \; , 
\end{eqnarray*}
where the second equality uses the index change $\ell = n-1-k$, and
for the third one we note that $(-1)^\ell = (-1)^{-\ell}$. This shows
that $v \in W_b$.

Finally, let us assume that $u \in W_c$. Then one has $-u(s) = T^n_n u(s)
= (-1)^n u(s+1)$ for all $s \in \R$. Thus, if we set $s = -x$, then we
obtain
\begin{displaymath}
  T_n^n v(x) = (-1)^n v \left( x + 1 \right) =
  (-1)^n u \left( 2-x-1 \right) = (-1)^n u(-x+1) =
  -u(-x)  = -v(x) \; ,
\end{displaymath}
which in turn implies $v \in W_c$. This completes the proof of the lemma.
\end{proof}
\begin{remark}[Dihedral group~$D_{2n}$ action]
While our main focus so far has been to understand how the action of
the cyclic group~$\Z_{2n}$ induced by~$T_n$ on $W = H^2_{per}(\R)$ or
$W = L^2_{per}(\R)$ can be used to find a suitable space decomposition,
Lemma~\ref{lem:symreflect} illustrates another point. In addition to
the action of~$T_n$, our study also makes crucial use of the reflection
symmetry~$u(\cdot) \mapsto u(-\cdot)$ on the space~$W$. Thus, the actual
underlying symmetry group is the resulting dihedral group~$D_{2n}$
acting on~$W$.
\end{remark}
The above result shows that the spaces in the decomposition of~$W$ are
invariant under the reflection~$x \mapsto -x$. This leads us immediately
to the following result, which further decomposes every~$W_\tau$ into
even and odd functions.
\begin{lemma}[Even and odd decomposition]
\label{lem:evenoddWdecomp}
Let $W = H^2_{per}(\R)$ or $W = L^2_{per}(\R)$, and define the 
two subspaces $W^e = \{ u \in W : \mbox{$u$ is even} \}$ and
$W^o = \{ u \in W : \mbox{$u$ is odd} \}$ of~$W$ consisting of
all even and odd functions, respectively. Then we have $W =
W^e \oplus W^o$, as well as
\begin{displaymath}
  W_\tau =
    \left( W_\tau \cap W^e \right) \oplus
    \left( W_\tau \cap W^o \right)
  \qquad\mbox{ for all }\qquad
  \tau \in \{ a,b,c \} \; .
\end{displaymath}
Finally, functions in~$W^e$ are even with respect to both~$x = 0$
and with respect to~$x = 1$, while functions in~$W^o$ are odd
with respect to both of these $x$-values.
%
%
%
\end{lemma}
\begin{proof}
It is well-known that every function~$u$ defined on~$\R$ can be
written as the sum of an even and an odd function in the form
$u = u^e + u^o$, where the even and odd parts are explicitly
given by $u^e(x) = (u(x)+u(-x))/2$ and $u^o(x) = (u(x)-u(-x))/2$,
respectively. Thus, in view of Lemma~\ref{lem:symreflect} we have
both~$u^e \in W_\tau$ and $u^o \in W_\tau$, as long as $u \in W_\tau$.
Since only the zero function is both even and odd, this implies
the decompositions stated in the lemma. The statement concerning
the evenness of every $u \in W^e$ with respect to~$x = 1$ has already
been shown in the verification of~\eqref{eqn:xsubw}. Finally, for
$v \in W^o$ one obtains
\begin{displaymath}
  v(1-x) = -v(x-1) = -v(x-1+2) = -v(1+x)
  \quad\mbox{ for all }\quad
  x \in \R \; ,
\end{displaymath}
and this completes the proof of the lemma.
\end{proof}

\medskip
As we already showed in~\eqref{eqn:xsubw}, the space~$X$
defined in~\eqref{def:xyspaces} can be considered as a subspace
of~$W = H^2_{per}(\R)$ in the sense that $u \in X$ if and only if
its extension~$\tilde{u}$ is an even function in~$H^2_{per}(\R)$.
Thus, the above lemma allows us to pull the space decomposition
defined in~\eqref{eqn:wabc} down to the space~$X$, by considering
only the even functions in the spaces~$W_a$, $W_b$, and~$W_c$.
More precisely, we have the following definition.
\begin{definition}[Symmetry induced space decomposition of~$X$]
\label{def:abcx}
Let $W = H^2_{per}(\R)$ denote the space defined
in~\eqref{def:wperspaces}, and let~$X$ be defined as
in~\eqref{def:xyspaces}. Then we define three subspaces
of~$X$ by considering only the even corresponding functions
in the subspaces~$W_a$, $W_b$, and~$W_c$ defined in~\eqref{eqn:wabc}
for some integer~$n \in \N$, i.e., we set 
\begin{displaymath}
  X_\tau = \left\{ u \in X \; : \;
    \tilde{u} \in W_\tau \right\}
  \quad\mbox{ for all }\quad
  \tau \in \{ a,b,c \}
\end{displaymath}
in view of~\eqref{eqn:xsubw}. Notice that
Lemma~\ref{lem:evenoddWdecomp} immediately implies
$X = X_a \oplus X_b \oplus X_c$.
\end{definition}
With the above definition we have achieved our first goal, namely,
the derivation of a decomposition of our domain~$X$ that is in some
sense induced by the symmetry~$T_n$, and that allows us to discuss
symmetry-breaking pitchfork bifurcations. Based on our derivation,
one would suspect that the pitchfork bifurcation equilibria shown
in Figures~\ref{fig:oddodd} or~\ref{fig:eveneven} are contained
in the spaces~$X_a$ --- and we still need to understand why the
spaces~$X_b$ and~$X_c$ are the correct spaces to include the
eigenfunctions. We would like to emphasize one more time, however,
that while in some sense~$X_a$ is invariant under the
symmetry~$T_n$ (via corresponding functions), this is not true
for the spaces~$X_b$ and~$X_c$.
\begin{table}
  \begin{center}
  \begin{tabular}{|c|c|c|} \hline
    & $v \in X_a \oplus X_b$ & $v \in X_c$ \\ \hline
    $n$ even & $v$ even with respect to $x=1/2$ & 
      $v$ odd with respect to $x=1/2$ \\[0.5ex]
    $n$ odd & $v$ odd with respect to $x=1/2$ & 
      $v$ even with respect to $x=1/2$ \\ \hline
  \end{tabular}
  \end{center}
  \caption{\label{table:wtauaddsym}
    Additional symmetries of functions in the spaces~$X_a$, $X_b$,
    and~$X_c$ introduced in Definition~\ref{def:abcx}. Depending on
    whether the underlying integer~$n \in \N$ is even or odd, functions
    in the spaces $X_a \oplus X_b$ and~$X_c$ have an additional even or
    odd symmetry with respect to the center point $x = 1/2$ of the domain
    $\Omega = (0,1)$, as listed in the above table.}
\end{table}

In addition to the symmetry properties discussed so far, functions
in the spaces~$X_a$, $X_b$, and~$X_c$ introduced in Definition~\ref{def:abcx}
exhibit one more symmetry. This is the subject of the following simple lemma.
\begin{lemma}[Symmetry with respect to $x = 1/2$]
\label{lem:symmx12}
Consider the spaces~$X_a$, $X_b$, and~$X_c$ introduced in
Definition~\ref{def:abcx}. Then functions in the direct
sum~$X_a \oplus X_b$ are even or odd with respect to the
center point $x = 1/2$ of the domain~$\Omega = (0,1)$ if the 
integer~$n \in \N$ is even or odd, respectively. In addition,
functions in~$X_c$ are odd or even with respect to $x = 1/2$
if~$n$ is even or odd, respectively. This is summarized in
Table~\ref{table:wtauaddsym}.
\end{lemma}
\begin{proof}
Consider first the case $v \in X_a \oplus X_b$, and let~$\tilde{v}$
denote its corresponding function in~$W$, which according to
Definition~\ref{def:abcx} and~\eqref{eqn:wabc} is contained
in~$W_a \oplus W_b = N(T_n^n - I)$. Thus, the extension~$\tilde{v}$
satisfies the identity $T_n^n \tilde{v} = \tilde{v}$, and by iterating
the definition of~$T_n$ one can easily see that this is equivalent to
the identity $(-1)^n \tilde{v}(1 + x) = \tilde{v}(x)$ for all $x \in \R$.
If we then replace~$x$ by~$x - 1/2$, this immediately implies
\begin{displaymath}
  (-1)^n \tilde{v} \left( \frac12 + x \right) =
    \tilde{v} \left( x -\frac12 \right) =
    \tilde{v} \left( \frac12 - x \right)
  \quad\mbox{ for all }\quad
  x \in \R \; ,
\end{displaymath}
where for the second identity we use the fact that~$\tilde{v}$
is even. This establishes the first half of the lemma. By a completely
analogous argument, using the fact that for $v \in X_c$ its extension
satisfies $\tilde{v} \in W_c = N(T_n^n + I)$, one can easily verify the
second half as well. All that changes is the introduction of an additional
negative sign, which is responsible for the switch between even and odd in
this case.
\end{proof}

\medskip
We would like to point out explicitly that, in view of the above lemma,
for any given integer~$n \in \N$ either $X_a \oplus X_b$ contains even
functions with respect to~$x = 1/2$ and~$X_c$ contains odd ones, or
vice versa. We will see later that this fact is inherently responsible
for the $\Z_2$ symmetry-breaking results of~\cite{lessard:sander:wanner:17a}.
As it turns out, the further decomposition into~$X_a$ and~$X_b$ allows us
to treat the bifurcation points shown in Figures~\ref{fig:oddodd}
and~\ref{fig:eveneven}.

As our final result concerning the decomposition of the space~$X$ we
now show that the spaces~$X_a$, $X_b$, and~$X_c$ are pairwise orthogonal.
In fact, the following result implies that they are orthogonal with respect
to a variety of possible inner products on~$X$.
\begin{lemma}[Orthogonality of the $X$-decomposition]
\label{lem:pairwiseorthogonal}
The spaces~$X_a$, $X_b$, and~$X_c$ introduced in Definition~\ref{def:abcx}
are pairwise orthogonal with respect to the inner product
\begin{displaymath}
  (\Delta u, \Delta v)_{L^2(\Omega)} +
    \beta (\nabla u, \nabla v)_{L^2(\Omega)} +
    \gamma (u,v)_{L^2(\Omega)}
  \qquad\mbox{ for }\qquad
  u,v \in X \; ,
\end{displaymath}
for any choice of constants $\beta \ge 0$ and $\gamma \ge 0$.
\end{lemma}
\begin{proof}
Assume that the functions~$u$ and~$v$ are taken from two different spaces
of~$X_a$, $X_b$, and~$X_c$, and let~$\tilde{u}$ and~$\tilde{v}$ denote their
extensions in~$W$.

We begin by showing that the standard $L^2(\Omega)$-inner
product of~$u$ and~$v$ vanishes. In view of Lemma~\ref{lem:symmx12}, this
is trivially satisfied if $u \in X_a \cup X_b$ and $v \in X_c$, or vice
versa, since in these cases the product~$uv$ is always odd with respect
to $x = 1/2$, and therefore $\int_0^1 u(x) v(x) \, dx = 0$. Assume 
therefore that $u \in X_a$ and $v \in X_b$. Then one obtains
\begin{eqnarray*}
  \int_0^1 u(x) v(x) \, dx & = &
    \sum_{k=0}^{n-1} \int_{\frac{k}{n}}^{\frac{k+1}{n}} u(x) v(x) \, dx
    \;\; = \;\;
    \sum_{k=0}^{n-1} \int_{0}^{\frac{1}{n}} u \left( x + \frac{k}{n} \right)
    v \left( x + \frac{k}{n} \right) \, dx \\
  & = & 
    \sum_{k=0}^{n-1} \int_{0}^{\frac{1}{n}}
    \left( (-1)^k u \left( x + \frac{k}{n} \right) \right)
    \left( (-1)^k v \left( x + \frac{k}{n} \right) \right) \, dx \\
  & = &
    \sum_{k=0}^{n-1} \int_{0}^{\frac{1}{n}}
    \left( T_n^k \tilde{u}(x) \right)
    \left( T_n^k \tilde{v}(x) \right) \, dx
    \;\; = \;\;
    \sum_{k=0}^{n-1} \int_{0}^{\frac{1}{n}}
    \tilde{u}(x) \, T_n^k \tilde{v}(x) \, dx \\
  & = &
    \int_{0}^{\frac{1}{n}} u(x)
    \left( \sum_{k=0}^{n-1} T_n^k \tilde{v}(x) \right) \, dx
    \;\; = \;\; 0 \; , 
\end{eqnarray*}
where we used the facts that $T_n^k \tilde{u} = \tilde{u}$
and $\sum_{k=0}^{n-1} T_n^k \tilde{v} = 0$.

The remaining two terms in the inner products defined in the formulation
of the lemma can be discussed similarly. The statement for the second
derivative term~$(\Delta u, \Delta v)_{L^2(\Omega)}$ follows completely
analogously, since the Laplacian~$\Delta$ commutes with~$T_n$ and
preserves any even or odd symmetry with respect to $x = 1/2$. Finally,
by using integration by parts and the boundary conditions imposed
in~$X$, one can obtain $(\nabla u, \nabla v)_{L^2(\Omega)} =
-(\Delta u, v)_{L^2(\Omega)}$, and the result follows again as before.
\end{proof}

\medskip
With the above result we have completed the construction of a
decomposition of the Hilbert space~$X$ defined in~\eqref{def:xyspaces},
which serves as the domain of our nonlinear diblock copolymer
operator~$F : X \to Y$ introduced in~\eqref{def:Flambdau}. We now
turn our attention to the image space~$Y$.
\begin{definition}[Symmetry induced space decomposition of~$Y$]
\label{def:abcy}
Consider the Hilbert space~$Y = H^{-2}(\Omega)$ for $\Omega = (0,1)$
introduced in~\eqref{def:xyspaces}. Let  $u$ be any element of $Y$. 
Then its inverse image~$\Delta^{-1} u$ with respect to the Laplacian operator
$\Delta: L^2(\Omega) \cap \left\{ \int_\Omega u \, dx = 0 \right\} \to
H^{-2}(\Omega)$ has an extension~$\hat{u} \in L^2_{per}(\R)$, given
by~$\hat{u} = \widetilde{\Delta^{-1} u}$. Furthermore, since the
Laplacian is an isometry with respect to our chosen norms, we have
with $W = L^2_{per}(\R)$ the decomposition
\begin{displaymath}
  L^2_{per}(\R) \; = \;
   W_a \oplus  W_b \oplus W_c \; ,
\end{displaymath}
where~$W_a$, $W_b$, and~$W_c$ contained in $L^2_{per}(\R)$ were defined in~\eqref{eqn:wabc}.
The subspaces in this decomposition are pairwise orthogonal. Thus, if we define
\begin{displaymath}
  Y_\tau = \left\{ u \in Y \; : \;
    \hat{u} \in  W_\tau  \right\}
  \quad\mbox{ for all }\quad
  \tau \in \{ a,b,c \}
\end{displaymath}
then one can verify that $Y = Y_a \oplus Y_b \oplus Y_c$,
and that the involved subspaces are pairwise orthogonal. In fact,
one can also show that $Y_\tau = \Delta^2 X_\tau$ for all
$\tau \in \{ a,b,c \}$. Finally, using our earlier definition of
$T_n : H^{-2}_{per}(\R) \to H^{-2}_{per}(\R)$ via the identity
$T_n(u) = \Delta T_n(\Delta^{-1}u )$, one can  verify that we
have $u \in Y_\tau$ if and only if one has the equality
$m_\tau(T_n)[\Delta\hat{u}] = 0$. In the following, we will therefore
call~$\Delta\hat{u} \in H^{-2}_{per}(\R)$ the extension of~$u \in Y$.
\end{definition}
The above descriptions of the spaces~$X$ and~$Y$, and of their respective
decompositions, are tailor-made for the results discussed in the remainder
of the paper. Nevertheless, we close this subsection with an alternative
description of these spaces, which is based on cosine Fourier series.
\begin{remark}[Cosine Fourier series representations]
\label{rem:fourierabc}
Consider an arbitrary function~$u \in X$. Then it was shown
in~\cite{sander:wanner:21a} that there exists a unique representation
\begin{equation} \label{rem:fourierabc1}
  u(x) \; = \; \sum_{k \in \N} a_k \cos(k \pi x)
  \quad\mbox{ with }\quad
  a_k \in \R
  \;\;\mbox{ for }\;\; k \in \N \; ,
\end{equation}
where the constant term $k = 0$ is omitted due to the zero mass 
constraint, and the series converges with respect to our chosen
norm on~$X$.

As it turns out, each of the three subspaces~$X_a$, $X_b$,
and~$X_c$ have a simple description in terms of this series 
representation, since the basis functions are pairwise orthogonal,
and every one of these functions is contained in exactly one of
these spaces.

To begin with, consider the space~$X_a$. For any function~$u \in X_a$,
if we denote its extension again by~$\tilde{u}$, one can easily see 
that in view of $T_n(\tilde{u}) = \tilde{u}$ we have
\begin{equation} \label{Xafunctionoddat1o2n}
  u\left( \frac{1}{2n} + x \right) \; = \;
  \tilde{u}\left( -\frac{1}{2n} - x \right) \; = \;
  -\tilde{u}\left( -\frac{1}{2n} - x + \frac{1}{n} \right) \; = \;
  -u\left( \frac{1}{2n} - x \right) \; ,
\end{equation}
for all $x \in (0,1 / (2n))$, i.e., the function~$u$ is odd with
respect to~$x = 1/(2n)$ on the subinterval~$(0,1/n)$. This in turn
implies that the average of~$u$ vanishes over~$(0,1/n)$, and since
the functions~$\cos(k\pi x/L)$ for~$L = 1/n$ and $k \in \N_0$ form
a complete orthogonal set in~$L^2((0,1/n))$, one obtains that
\begin{displaymath}
  u \in X_a
  \quad\mbox{ if and only if }\quad
  u(x) \; = \; \sum_{k \in n\N} a_k \cos(k \pi x) 
    \; = \; \sum_{\ell \in \N} a_{\ell n} \cos(\ell n \pi x) \; ,
\end{displaymath}
since the basis functions~$\cos(\ell n \pi x)$ for $\ell \in \N$ clearly
lie in~$X_a$ themselves. In other words, the function~$u$ is contained
in~$X_a$ if and only if its cosine Fourier series contains only terms
corresponding to wave numbers~$k$ which are multiples of~$n$.

We now turn our attention to the spaces~$X_b$ and~$X_c$. Due to
Lemma~\ref{lem:pairwiseorthogonal}, cosine Fourier expansions of
a function~$u$ in either of these spaces can only contain terms
for wave numbers which are not divisible by~$n$. Furthermore, one
can verify by direct inspection that
\begin{displaymath}
  \begin{array}{ccrcc}
    \DS T_n^n \cos(k \pi x) & = & \DS \cos(k \pi x) &
      \quad\mbox{ if and only if }\quad &
      \DS n \equiv k \mod 2 \; , \\[1ex]
    \DS T_n^n \cos(k \pi x) & = & \DS -\cos(k \pi x) &
      \quad\mbox{ if and only if }\quad &
      \DS n \not\equiv k \mod 2 \; .
  \end{array}
\end{displaymath}
This gives the characterizations
\begin{displaymath}
  \begin{array}{lccccl}
    \mbox{Even~$n$:} & \DS u \in X_b &
      \quad\mbox{ if and only if }\quad & \DS u(x) & = & \DS
      \sum_{k \not\in n\N \, , \; k \;\rm{even}} a_k \cos(k \pi x)
      \; , \\[4ex]
    & \DS u \in X_c &
      \quad\mbox{ if and only if }\quad & \DS u(x) & = & \DS
      \sum_{k \not\in n\N \, , \; k \;\rm{odd}} a_k \cos(k \pi x)
      \; , \\[4ex]
    \mbox{Odd~$n$:} & \DS u \in X_b &
      \quad\mbox{ if and only if }\quad & \DS u(x) & = & \DS
      \sum_{k \not\in n\N \, , \; k \;\rm{odd}} a_k \cos(k \pi x)
      \; , \\[4ex]
    & \DS u \in X_c &
      \quad\mbox{ if and only if }\quad & \DS u(x) & = & \DS
      \sum_{k \not\in n\N \, , \; k \;\rm{even}} a_k \cos(k \pi x)
      \; .
  \end{array}
\end{displaymath}
We would like to point out that these characterizations immediately
imply the statements of Table~\ref{table:wtauaddsym}. Furthermore,
the above characterizations remain valid without change for the
image space decomposition~$Y = Y_a \oplus Y_b \oplus Y_c$, as long
as one considers the cosine Fourier series in a formal sense and its
convergence in the norm defined in~$Y$. For more details, we refer
the reader to~\cite{sander:wanner:21a}.
\end{remark}
%
%
\subsection{Equivariance properties of the nonlinear operator}
\label{subsec:symmetry3}
We close this section by establishing the equivariance properties of the
nonlinear diblock copolymer operator~$F : X \to Y$ defined in~(\ref{def:Flambdau})
and~(\ref{def:xyspaces}), with a particular emphasis on how this operator and its
Fr\'echet derivative~$L$ defined in~(\ref{eqn:dfu}) interacts with the space
decompositions of~$X$ and~$Y$ from the last subsection. As mentioned in the
introduction, throughout this paper we consider the one-dimensional domain
$\Omega = (0,1)$, the total mass $\mu = 0$, and the odd nonlinearity
$f(u) = u - u^3$. We would like to point out, however, that the results
of this section remain valid for any smooth odd function~$f : \R \to \R$,
and are in fact formulated for that case.

Throughout this section, we consider mapping properties of operators between
the spaces~$X_\tau$ and~$Y_\tau$ introduced in the last section. In order to
keep the notation as simple as possible, we will use the same letter for a
function in~$X_\tau$ and its extension in~$H^2_{per}(\R)$, and similarly
for elements in~$Y_\tau$ and their extensions in~$H^{-2}_{per}(\R)$. Thus,
we can consider the diblock copolymer operator~$F$ both as an operator
between~$X$ and~$Y$, as well as an operator of the form $F : H^2_{per}(\R)
\to H^{-2}_{per}(\R)$.

It has already been stated several times that our main focus is the verification
of a special kind of symmetry-breaking pitchfork bifurcation. Thus, the equilibrium
solutions on the bifurcating branch will exhibit different, and in fact fewer,
symmetry properties than the solutions on the primary steady state branch. This
primary branch is characterized by invariance with respect to the symmetry
operator~$T_n$ defined in~(\ref{eqn:z2nsymm}), and the following first lemma
shows that both~$F$ and its partial derivative~$D_{\lambda}F$
respect this symmetry.
\begin{lemma}[First equivariance properties]
\label{lem:xapreserved}
Let $\Omega = (0,1)$, consider the total mass $\mu = 0$, and let~$f$ be a smooth
and odd nonlinearity. Furthermore, let~$F : X \to Y$ be defined as in~(\ref{def:Flambdau})
and~(\ref{def:xyspaces}). Then for every $u \in X_a$ we have both $F(\lambda,u) \in Y_a$
and $D_\lambda F(\lambda,u) \in Y_a$.
\end{lemma}
\begin{proof}
Let~$u \in X_a$ be arbitrary. Then its extension in~$H^2_{per}(\R)$ is even, and due
to the properties of the Laplacian and the Nemitski operator~$f$ the same is true for
both~$F(\lambda,u)$ and~$D_\lambda F(\lambda,u)$. Moreover, the oddness of~$f$ 
and $T_n u = u$ imply
\begin{displaymath}
  T_n f(u(x)) = -f(u(x+1/n)) = f(-u(x+1/n)) = f(u(x)) \; .
\end{displaymath}
In combination with $T_n \Delta = \Delta T_n$ one therefore obtains 
$T_n F(\lambda,u) = F(\lambda,u)$, and this in turn implies $F(\lambda,u) \in Y_a$,
see also Definition~\ref{def:abcy}. The inclusion for
$D_\lambda F(\lambda,u) = -\Delta f(u) - \sigma u$ can be verified analogously.
\end{proof}

\medskip
In view of this lemma, one can study the equilibrium problem~$F(\lambda,u) = 0$
restricted to the symmetry space~$X_a$, and this will provide us with a primary 
solution branch in~$X_a$. The next two auxiliary results address how first-
and second-order partial derivatives of~$F$ interact with the various symmetry
spaces. The resulting inclusions are central for our bifurcation analysis.
\begin{lemma}[Equivariance properties of $D_uF$ and $D_{\lambda u}F$]
\label{lem:subsets}
Let $\Omega = (0,1)$, consider the total mass $\mu = 0$, and let~$f$ be a smooth
and odd nonlinearity. Furthermore, let~$F : X \to Y$ be defined as in~(\ref{def:Flambdau})
and~(\ref{def:xyspaces}). Then for arbitrary $\lambda \in  \R$ and $u \in X_a$, and
every $\tau \in \{ a,b,c \}$ we have the inclusions 
\begin{equation} \label{lem:subsets1}
  D_uF(\lambda,u)[X_\tau] \subset Y_\tau
  \qquad\mbox{ and }\qquad
  D_{\lambda u}F(\lambda,u)[X_\tau] \subset Y_\tau \; .
\end{equation}
In addition, we have
\begin{equation} \label{lem:subsets2}
  R(D_uF(\lambda,u)) =
    D_u F(\lambda,u)[X_a] \oplus
    D_u F(\lambda,u)[X_b] \oplus
    D_u F(\lambda,u)[X_c] \; ,
\end{equation}
as well as both
\begin{equation} \label{lem:subsets3}
  D_u F(\lambda,u)[X_\tau] =
  R(D_u F(\lambda,u)) \cap Y_\tau
  \qquad\mbox{ and }\qquad
  P(Y_\tau) \subset Y_\tau
\end{equation}
for all $\tau \in \{ a,b,c \}$, where $P : Y \to Y$ denotes the orthogonal
projection which was introduced in Definition~\ref{def:orthproj}.
\end{lemma}
\begin{proof}
Let~$\lambda \in \R$ and $u \in X_a$ be fixed, and consider an arbitrary~$v \in X$.
For notational convenience in this proof, we introduce the abbreviation $L = D_u F(\lambda,u)$.
Since we assumed that~$f$ is an odd function, its derivative~$f'$ is even. This
immediately implies
\begin{eqnarray*}
  T_n \left( f'(u(x)) v(x) \right) & = & -f'(u(x+1/n)) v(x+1/n)  \\
  & = & f'(-u(x)) (- v(x+1/n))  \; = \;
  f'(u(x)) T_n v (x) \; .
\end{eqnarray*}
Since the operator~$T_n$ also commutes with each of the other two terms
in the explicit representation~(\ref{eqn:dfu}) of~$D_uF(\lambda,u)[v]$,
one therefore obtains $T_n L [v] = L [T_n v]$. This in turn implies that
the inclusion~$v \in X_\tau$ automatically implies $D_uF(\lambda,u)[v] \in Y_\tau$,
where we again refer the reader to Definition~\ref{def:abcy}. The statement
for $D_{\lambda u} F(\lambda,u)[v] = -\Delta(f'(u)v) - \sigma v$ can be established
completely analogously, and this completes the proof of~(\ref{lem:subsets1}).

The just-established~(\ref{lem:subsets1}) implies that the spaces~$L[X_a]$, $L[X_b]$, and~$L[X_c]$
form a direct sum, since the spaces~$Y_a$, $Y_b$, and~$Y_c$ do. We clearly
have $L[X_a] \oplus L[X_b] \oplus L[X_c] \subset R(L)$, and the reverse inclusion
follows from $R(L) \ni y = Lx = Lx_a + Lx_b + Lx_c$, if we write
$x = x_a + x_b + x_c \in X_a \oplus X_b \oplus X_c = X$. This
implies~(\ref{lem:subsets2}).

As for~(\ref{lem:subsets3}), let~$\tau \in \{ a,b,c \}$ be arbitrary. Then
one obviously has $L[X_\tau] \subset R(L) \cap Y_\tau$. To verify the opposite
inclusion, let~$y \in R(L) \cap Y_\tau$ be arbitrary. Then $y = Lx$, and we can
again decompose~$x$ in the form $x = x_a + x_b + x_c \in X_a \oplus X_b \oplus X_c$.
But this in turn yields the inclusion $y = Lx_a + Lx_b + Lx_c \in Y_a \oplus Y_b
\oplus Y_c$, and $y \in Y_\tau$ in combination with the properties of direct sums
immediately imply $y = Lx_\tau \in L[X_\tau]$. In order to establish the 
inclusion statement regarding the orthogonal projection~$P$, one just has to
note that every $y \in Y$ can be written uniquely as $y = y_a + y_b + y_c \in
Y_a \oplus Y_b \oplus Y_c$, and that for $\tau \in \{ a,b,c \}$ one further
has $y_\tau = y_{\tau,1} + y_{\tau,2}$, where~$y_{\tau,1} \in L(X_\tau)$
and~$y_{\tau,2}$ is contained in the orthogonal complement of~$L(X_\tau)$
in~$Y_\tau$, which might in fact be trivial. Altogether, this gives the
decomposition
\begin{displaymath}
  y = y_{a,1} + y_{a,2} + y_{b,1} + y_{b,2} + y_{c,1} + y_{c,2}
\end{displaymath}
into pairwise orthogonal elements, and one can easily see that
$Py = y_{a,2} + y_{b,2} + y_{c,2}$. From this, the last statement follows
readily, and the proof of the lemma is complete.
\end{proof}
\begin{lemma}[Equivariance properties of~$D_{uu}F$]
\label{lem:derivcontain}
Let $\Omega = (0,1)$, consider the mass $\mu = 0$, and let~$f$ be a smooth
and odd nonlinearity. Furthermore, let~$F : X \to Y$ be defined as in~(\ref{def:Flambdau})
and~(\ref{def:xyspaces}). Then for all $\lambda \in \R$ and $u \in X_a$ the following
inclusions are satisfied:
\begin{equation}
  \begin{array}{ccccc}
    \mathit{(i)}   & \quad & D_{uu}F(\lambda,u)[X_a,X_a] &\subset& Y_a\\[0.5ex]
    \mathit{(ii)}  & & D_{uu}F(\lambda,u)[X_a \oplus X_b,X_a \oplus X_b] &\subset& Y_a \oplus Y_b\\[0.5ex]
    \mathit{(iii)} & & D_{uu}F(\lambda,u)[X_a,X_b] &\subset& Y_b \\[0.5ex]
    \mathit{(iv)}  & & D_{uu}F(\lambda,u)[X_c,X_c] &\subset& Y_a \oplus Y_b\\[0.5ex]
    \mathit{(v)}   & & D_{uu}F(\lambda,u)[X_a \oplus X_b,X_c] &\subset& Y_c
  \end{array}
\end{equation}
Note that due to the symmetry of~$D_{uu} F(\lambda,u)$, the order
of the two arguments in~$\mathit{(iii)}$ and~$\mathit{(v)}$ does not matter.
\end{lemma}
\begin{proof}
Recall that we have $D_{uu}F(\lambda,u)[v,w] = -\Delta (\lambda f''(u) v w)$ for
all $v,w \in X$, and that due to the oddness of~$f$ the second derivative~$f''$ is
also an odd function. Then for~$u \in X_a$ one obtains the identity
\begin{eqnarray*}
  T_n \left( f''(u(x)) v(x) w(x) \right) &=& -f''(u(x+1/n)) v(x+1/n) w(x+1/n) \\
  &=& f''(-u(x+1/n)) (-v(x+1/n))(-w(x+1/n)) \\
  &=& f''(u(x)) T_n v(x) T_n w(x) \;,
\end{eqnarray*}
which in turn readily implies $T_n D_{uu}F(\lambda,u)[v,w] = D_{uu}F(\lambda,u)[T_n v,T_n w]$.
With this formula at hand one can now establish the claims.

Note that if $v,w \in X_a$, then we have both $T_n v = v$ and $T_n w = w$,
and the first statement follows. Similarly, if $v,w \in X_a \oplus X_b$, then
$T_n^n v = v$ and $T_n^n w = w$, which yields the second statement. In addition,
if we assume $v,w \in X_c$, then $T_n^n v = -v$ and $T_n^n w = -w$, and from this
one can obtain~{\em (iv)\/}.

We now turn our attention to~{\em (iii)\/}. If we have $v \in X_a$ and $w \in X_b$,
then for all $k \in \N$ one obtains $T_n^k D_{uu}F(\lambda,u)[v,w] =
D_{uu}F(\lambda,u)[T_n^k v,T_n^k w] = D_{uu}F(\lambda,u)[v,T_n^k w]$. But this
gives
\begin{displaymath}
  \sum_{k = 0}^{n-1} T_n^k D_{uu}F(\lambda,u)[v,w] =
  D_{uu}F(\lambda,u)\left[v, \sum_{k = 0}^{n-1} T_n^k w \right] = 0 \; ,
\end{displaymath}
which in turn implies $D_{uu}F(\lambda,u)[v,w] \in Y_b$. This establishes~{\em (iii)\/}.

Finally, suppose that $v \in X_a \oplus X_b$ and $w \in X_c$. Then $T_n^n v = v$
and $T_n^n w = -w$, and therefore $T_n^n D_{uu}F(\lambda,u)[v,w] =
D_{uu}F(\lambda,u)[T_n^n v,T_n^n w] = -D_{uu}F(\lambda,u)[v,w]$, which 
yields~{\em (v)\/}. This completes the proof of the lemma.
\end{proof}

\medskip
After these preparations we can now turn our attention to the main result of
this section. It shows that under our Hypotheses~\ref{hyp:fred} and~\ref{hyp:1dk}
the kernel function at a potential bifurcation point has to be contained in
one of the spaces~$X_a$, $X_b$, and~$X_c$. In addition, we obtain easily testable
conditions that establish the precise space containing the kernel function. This
will be essential for the bifurcation result and its rigorous verification via
extended systems in the next section.
\begin{proposition}[Kernel function properties]
\label{prop:phi0}
Let $\Omega = (0,1)$, consider the mass $\mu = 0$, and let the function~$f$ be
a smooth and odd nonlinearity. Furthermore, let~$F : X \to Y$ be defined as
in~(\ref{def:Flambdau}) and~(\ref{def:xyspaces}), and suppose that
Hypotheses~\ref{hyp:fred} and~\ref{hyp:1dk} are satisfied, that is, we have
both $F(\lambda_0,u_0) = 0$ and $L\phi_0 = D_u F(\lambda_0,u_0)[\phi_0] = 0$,
and the null space of~$L$ is one-dimensional. In addition, we assume that
the equilibrium~$u_0$ is contained in the symmetry space~$X_a$. Then the
following statements hold.
\begin{enumerate}
\item[(a)] The kernel function~$\phi_0$ is automatically contained in
either~$X_a$, or~$X_b$, or~$X_c$.
\item[(b)] By verifying one of the conditions in the left or right columns
of Table~\ref{table:phi0force}, one can easily establish whether~$\phi_0 \in X_b$
or $\phi_0 \in X_c$ is satisfied, respectively.
\item[(c)] If the kernel function satisfies $\phi_0 \notin X_a$, then the
linearization $L : X_a \to Y_a$ is bijective. In particular, we
have $L[X_a] = Y_a$ in this case.
\end{enumerate}
\end{proposition}
\begin{table}
  \begin{center}
  \begin{tabular}{ |c|c|c| }
    \hline
    & $\phi_0 \in X_b$ & $\phi_0 \in X_c$ \\
    \hline
    $n$ even & $\phi_0\left( \frac{1}{2n} \right) \neq 0$, \;
    $\phi_0(0)+\phi_0(1) \neq 0$ & $\phi_0(0)-\phi_0(1) \neq 0$ \\[1ex]
    $n$  odd & $\phi_0\left( \frac{1}{2n} \right) \neq 0$, \;
    $\phi_0(0)-\phi_0(1) \neq 0$ & $\phi_0(0)+\phi_0(1) \neq 0$ \\
    \hline
  \end{tabular}
  \caption{Easily verifiable conditions which ensure whether the kernel
           function~$\phi_0$ is contained in~$X_b$ or in~$X_c$.
           These conditions change depending on whether~$n$ in the 
           definition of the symmetry operator~$T_n$ is even or odd.
           Since all of these conditions are simple inequality checks
           for specific function values, they can readily be rigorously validated
           using interval arithmetic.
           \label{table:phi0force}}
  \end{center}
\end{table}
\begin{proof}
We begin by verifying~{\em (a)\/}. Since~$\phi_0 \in X$, we can find
$\phi_\tau \in X_\tau$ for $\tau = a,b,c$ such that the identity
$\phi_0 = \phi_a + \phi_b + \phi_c$ is satisfied. But then~(\ref{lem:subsets1})
implies $L[\phi_\tau] \in Y_\tau$, i.e., we have $0 = L[\phi_0] =
L[\phi_a] + L[\phi_b] + L[\phi_c] \in Y_a \oplus Y_b \oplus Y_c$, and
the properties of direct sums therefore furnish $L[\phi_\tau] = 0$ for
all $\tau = a,b,c$. Since $\phi_0 \neq 0$, we have to have $\phi_\tau \neq 0$
for at least one $\tau \in \{ a,b,c \}$. Since the null space of~$L$ is
one-dimensional, one then has to have the equality $\phi_0 = \alpha \phi_\tau
\in X_\tau$ for some $\alpha \neq 0$, which establishes the claim.

Consider now the statement in~{\em (b)\/} and suppose for the moment that
the integer~$n$ in the definition of~$T_n$ is even. Then in view of
Lemma~\ref{lem:symmx12}, see also Table~\ref{table:wtauaddsym}, every
function in the direct sum~$X_a \oplus X_b$ is even with respect to
$x = 1/2$, while every function in~$X_c$ is odd with respect to the
center of the interval~$\Omega$. By taking the contrapositive, if we know
that the inequality $\phi_0(0)+\phi_0(1) \neq 0$ is true, then~$\phi_0$
cannot be odd with respect to $x = 1/2$, and therefore has to be in~$X_a$
or~$X_b$. Similarly, the inequality $\phi_0(0)-\phi_0(1) \neq 0$ shows
that~$\phi_0$ is not even, and so it has to be in~$X_c$. The case of
odd~$n$ can be treated similarly.

It remains to show that if we have the inclusion~$\phi_0 \in X_a \cup X_b$,
as well as $\phi_0(1/(2n)) \neq 0$, then necessarily one has $\phi_0 \in X_b$.
This, however, follows immediately from~(\ref{Xafunctionoddat1o2n}), where
it was demonstrated that every function in~$X_a$ is odd with respect
to~$x = 1/(2n)$, and therefore has to vanish at~$1/(2n)$. This completes
the proof of~{\em (b)\/}.

Finally, we turn our attention to the statement in~{\em (c)\/}. Suppose 
therefore that~$\phi_0$ is contained in either~$X_b$ or~$X_c$. Since we
assumed $T_n u_0 = u_0$, the function~$u_0$ satisfies homogeneous Neumann
boundary conditions at both $x=0$ and at $x=1/n$. Now let~$u_s$ denote the
restriction of~$u_0$ to the interval~$\Omega_s = (0,1/n)$. Then we still
have $F(\lambda_0, u_s) = 0$ on this interval, with the same Neumann boundary
conditions that were considered for the original equation on $\Omega = (0,1)$.
Define the function spaces~$X_s$ and~$Y_s$ as the spaces corresponding to~$X$
and~$Y$, but consisting of functions restricted to the smaller interval~$\Omega_s$,
which in the case of~$X_s$ satisfy homogeneous Neumann boundary conditions
on~$\partial\Omega_s$. Finally, define $L_s = D_u F(\lambda_0,u_s)|_{X_s}$,
i.e., via restriction to the interval~$(0,1/n)$. 

We begin by verifying $L[X_s]  = Y_s$. For this, assume that $Y_s \setminus L_s[X_s]
\ne \emptyset$. Then there exists a nontrivial element which is not in the range
of~$L_s$. Therefore, since also the restricted operator is Fredholm with index~$0$ according
to~\cite[Proposition~2.15]{lessard:sander:wanner:17a}, there has to be a nontrivial
null space element $\phi_s \in N(L_s)$. Recall that any function in~$X_a$ is
uniquely defined by its values on~$\Omega_s = (0,1/n)$. Thus, since~$\phi_s$ satisfies
homogeneous Neumann boundary conditions on~$\Omega_s$, there is a corresponding unique
function $\phi_a \in X_a$ defined on~$\Omega = (0,1)$ and such that $\phi_a = \phi_s$
on~$(0,1/n)$. Moreover, the fact that $L_s \phi_s = 0$ on~$\Omega_s$ immediately
implies that $L\phi_a = 0$ on~$\Omega$. Therefore, the function~$\phi_a \in X_a$
is contained in the null space~$N(L)$. However, this is a contradiction, since
we have assumed that~$N(L)$ is one-dimensional, and that it is spanned by a
function~$\phi_0$ which is contained in either~$X_b$ or~$X_c$. Thus we conclude
that $L[X_a] = Y_a$. Since the above argument also directly implies $N(L_s) = \{ 0 \}$,
this completes the proof of the proposition.
\end{proof}
%
%
%
%
\section{Cyclic Equivariant Pitchfork Bifurcations}
\label{sec:pitbif}
After the preparations of the last section, we now show that the cyclic action of the
symmetry operator~$T_n$, through its derived space decomposition $X = X_a \oplus X_b
\oplus X_c$, does indeed give rise to symmetry-breaking bifurcations. More precisely,
we consider the scenarios indicated in Table~\ref{table:bifscenarios}. In this table,
we distinguish between the parity of the integer~$n$ and the symmetry with respect to
$x = 1/2$ of the kernel function~$\phi_0$. This leads to four different bifurcation 
scenarios, all of which break the $\Z_{2n}$-symmetry of the equilibrium solution
$u_0 \in X_a$, based on whether one has $\phi_0 \in X_b$ or $\phi_0 \in X_c$. Recall
that the latter two conditions can easily be verified using the tests listed in
Table~\ref{table:phi0force}. While all of these scenarios are covered by the theory
developed in the present paper, cases~(b) and~(c) could already be established using
the results of~\cite{lessard:sander:wanner:17a}. However, the cases~(a) and~(d) are
new and do require our new approach, and they cover all of the situations shown in
Figures~\ref{fig:oddodd} and~\ref{fig:eveneven}.
\begin{table}
  \begin{center}
  \begin{tabular}{ |c||c|c| } 
    \hline
    & $\phi_0$ even & $\phi_0$ odd \\ [0.5ex]
    \hline\hline
    $n$ even & (a) $u_0$ even & (b) $u_0$ even \\
             & $\phi_0 \in X_b$ & $\phi_0 \in X_c$ \\
    \hline
    $n$ odd  & (c) $u_0$ odd & (d) $u_0$ odd \\
             & $\phi_0 \in X_c$ & $\phi_0 \in X_b$ \\
    \hline
  \end{tabular}
  \caption{Symmetry-breaking pitchfork bifurcation scenarios if the
           equilibrium solution~$u_0$ satisfies $u_0 \in X_a$, i.e., we
           have $T_n u_0 = u_0$. Throughout the table, the labels even
           and odd refer to symmetries with respect to the center point
           $x = 1/2$ of the domain $\Omega = (0,1)$. Notice that~(b)
           and~(c) are covered by our previous $\Z_2$-pitchfork bifurcation
           theorem~\cite{lessard:sander:wanner:17a}. In contrast, the 
           remaining two cases~(a) and~(d) correspond to the new scenarios
           depicted in Figures~\ref{fig:oddodd} and~\ref{fig:eveneven},
           respectively.
           \label{table:bifscenarios}}
  \end{center}
\end{table}

To develop this new approach, we proceed as follows. In Section~\ref{subsec:pitbif1}
we use a standard Lyapunov-Schmidt reduction based on our underlying space decomposition
to provide a sufficient condition for the existence of a $\Z_{2n}$-symmetry breaking
pitchfork bifurcation. After that, Section~\ref{subsec:pitbif2} demonstrates that 
the assumptions of this result can be verified using the existence of an isolated 
zero of a suitable extended system. This reformulation of the existence result
makes it amenable to verification via computer-assisted proof techniques.
\subsection{A sufficient condition for pitchfork bifurcation}
\label{subsec:pitbif1}
We begin by concentrating on the derivation of a sufficient condition
for the existence of a pitchfork bifurcation which breaks the $\Z_{2n}$-symmetry.
For this, we rely on the Lyapunov-Schmidt reduction result in
Proposition~\ref{prop:lyap} below, which gives a general method of reducing
the bifurcation problem from an infinite-dimensional problem to a bifurcation
problem on a one-dimensional subspace. Our formulation is completely analogous
to the one used in~\cite{lessard:sander:wanner:17a}. Thus, we refer the reader
to this paper for the full proof and merely provide a brief sketch below to
keep the current paper self-contained.
\begin{proposition}[Lyapunov-Schmidt reduction]
\label{prop:lyap}
Let $\Omega = (0,1)$, consider the mass $\mu = 0$, and let the function~$f$ be
a smooth and odd nonlinearity. Furthermore, let~$F : X \to Y$ be defined as
in~(\ref{def:Flambdau}) and~(\ref{def:xyspaces}), and suppose that
Hypotheses~\ref{hyp:fred} and~\ref{hyp:1dk} are satisfied, i.e., we have
both $F(\lambda_0,u_0) = 0$ and $L\phi_0 = D_u F(\lambda_0,u_0)[\phi_0] = 0$,
and the null space of~$L$ is one-dimensional. Finally, let~$P$ and~$Q$ denote
the orthogonal projections from Definition~\ref{def:orthproj}.

Then there exist a neighborhood~$\Lambda_0$ of~$\lambda_0$, a neighborhood~$V_0$
of~$v_0 = Q u_0  \in N(L)$, a smooth function $W : \Lambda_0 \times V_0 \to \tilde{X}$,
as well as a smooth real-valued function~$b$ which is defined in a neighborhood of
the point~$(\lambda_0,0) \in \R^2$ such that the following hold:
\begin{itemize}
\item[(a)] If~$(\lambda,\alpha)$ is sufficiently close to the
point~$(\lambda_0,0) \in \R^2$ and satisfies $b(\lambda,\alpha) = 0$,
then we have
\begin{displaymath}
  F(\lambda, u) = 0
  \quad\mbox{ for }\quad
  u = v_0 + \alpha \phi_0 + W\left( \lambda, v_0 + \alpha \phi_0 \right) \; .
\end{displaymath}
\item[(b)] Conversely, if~$(\lambda,u)$ is close enough
to~$(\lambda_0,u_0)$ and solves~$F(\lambda,u)=0$, then
for~$\alpha$ defined via $v_0 + \alpha \phi_0 = Qu$ we
have~$b(\lambda,\alpha) = 0$ and $u = Qu + W(\lambda,Qu)$.
\end{itemize}
In other words, the solution set of $b(\lambda,\alpha)=0$ in a neighborhood
of $(\lambda_0,0) \in \R^2$ is in one-to-one correspondence with the solution
set of $F(\lambda,u) = 0$ in a neighborhood of $(\lambda_0,u_0)$. 
\end{proposition}
\begin{proof}
Due to the properties of the projections~$P$ and~$Q$, if we associate with every
$u \in X$ the two elements $v = Qu$ and $w = (I-Q)u$, then we clearly have
$u = v + w \in N(L) \oplus \tilde{X}$. Moreover, the nonlinear problem $F(\lambda,u) = 0$
is equivalent to the system
\begin{equation} \label{eqn:lyapg}
  P F(\lambda,v+w) = 0
  \qquad\mbox{ and }\qquad
  G(\lambda,v,w) := (I-P) F(\lambda,v+w) = 0 \; .
\end{equation}
One can show that the function $G: \R \times N(L) \times \tilde{X} \to R(L)$ introduced
in the second equation of this system has an invertible Fr\'echet derivative
$D_wG(\lambda_0,v_0, u_0-v_0) \in \cL(\tilde{X},R(L))$, and therefore the implicit
function theorem implies that given~$(\lambda,v)$ near~$(\lambda_0,v_0)$, there exists a
unique $w = W(\lambda,v)$ in~$\tilde{X}$ such that $G(\lambda,v,W(\lambda,v)) = 0$.
Since~$N(L)$ is one-dimensional, each $v \in N(L)$ has a unique representation of the
form $v_0 + \alpha \phi_0$. Given a nontrivial~$\psi_0^* \in Y^*$ such that $R(L) =
N(\psi_0^*)$, the function~$b$ is then defined as
\begin{equation} \label{eqn:lyapb}
  b(\lambda,\alpha) =
  \psi_0^*( PF(\lambda, v_0 + \alpha \phi_0 + W(\lambda,v_0+\alpha \phi_0))) \; . 
\end{equation}
The statements of the proposition now follow easily from the fact that the
solutions of the original problem $F(\lambda,u) = 0$ are in one-to-one correspondence
with the solutions of~(\ref{eqn:lyapg}).
\end{proof}

\medskip
The above proposition is the standard version of the Lyapunov-Schmidt reduction,
which does not account for any symmetry properties of the nonlinear operator~$F$.
Note, however, that we do require that the two projections~$P$ and~$Q$ are
orthogonal projections, and in that way the results from the last section allow
us to make a number of additional deductions as long as we assume that the equilibrium
solution~$u_0$ is contained in~$X_a$, while the kernel function~$\phi_0$ is contained
in either~$X_b$ or~$X_c$. More precisely, we will soon see that the following hold:
\begin{itemize}
\item Since the spaces in the decomposition~$X = X_a \oplus X_b \oplus X_c$
are pairwise orthogonal, we automatically obtain $Q[X_a] = \{ 0 \}$.
\item In view of~$F(\lambda,X_a) \subset Y_a$ and the assumption on~$\phi_0$,
there exists a unique branch of equilibrium solutions through~$(\lambda_0,u_0)$
which is contained in~$X_a$, and the application of the projection~$Q$ transforms
this branch into a trivial solution branch in~$\R \times N(L)$.
\item The construction of the bifurcation equation~$b(\lambda,\alpha) = 0$ in
the above proposition then readily implies that~$b(\lambda,0) = 0$ for all~$\lambda$
in a neighborhood of~$\lambda_0$. In fact, since we also assumed the oddness
of the nonlinearity~$f$, we can even make statements about the vanishing of certain
derivatives of the function~$b$.
\end{itemize}
These observations are explained in more detail in the following main result
of this subsection. It provides conditions under which the $\Z_{2n}$-equivariance
of the last section forces a symmetry-breaking pitchfork bifurcation. This result
is similar in spirit to~\cite[Proposition~2.11]{lessard:sander:wanner:17a},
as well as to the classical result~\cite{crandall:rabinowitz:71a}.
\begin{theorem}[Existence of $\Z_{2n}$-symmetry breaking pitchfork bifurcation]
\label{thm:bifexist}
Let $\Omega = (0,1)$, consider the mass $\mu = 0$, and let the function~$f$ be
a smooth and odd nonlinearity. Furthermore, let~$F : X \to Y$ be defined as
in~(\ref{def:Flambdau}) and~(\ref{def:xyspaces}), and suppose that
Hypotheses~\ref{hyp:fred} and~\ref{hyp:1dk} are satisfied, i.e., we have
both $F(\lambda_0,u_0) = 0$ and $L\phi_0 = D_u F(\lambda_0,u_0)[\phi_0] = 0$,
and the null space of~$L$ is one-dimensional. Finally, let~$P$ and~$Q$ denote
the orthogonal projections from Definition~\ref{def:orthproj}.
In addition, suppose that $u_0 \in X_a$ and that $\phi_0 \in X_\tau$ for
$\tau \in \{ b,c \}$. Then there exists a unique function $\xi_0 \in X_a$
such that
\begin{equation} \label{thm:bifexist1}
  L\xi_0 + (I-P) D_\lambda F(\lambda_0,u_0) = 0 \; ,
\end{equation}
and if we further suppose that the nondegeneracy condition
\begin{equation} \label{thm:bifexist2}
  D_{\lambda u} F(\lambda_0,u_0)[ \phi_0 ] +
    D_{uu} F(\lambda_0,u_0)[\phi_0,\xi_0] \notin R(L)
\end{equation}
is satisfied, then the point~$(\lambda_0,u_0)$ is a pitchfork bifurcation
point for the nonlinear operator~$F$.
\end{theorem}
\begin{proof}
Consider the function~$G(\lambda,v,w)$ defined in~\eqref{eqn:lyapg}, as well
as~$b(\lambda,\alpha)$ introduced in~\eqref{eqn:lyapb}. Furthermore, let
$v_0 = Q u_0 \in N(L)$ as in Proposition~\ref{prop:lyap}. Then our assumption
$u_0 \in X_a$, combined with the fact that~$\phi_0 \in X_\tau$ for $\tau \neq a$,
that~$Q : X \to N(L)$ is an orthogonal projection, and the orthogonality statement
from Lemma~\ref{lem:pairwiseorthogonal}, immediately yield both~$v_0 = 0$ and
$N(L|_{X_a}) = \{ 0 \}$. Furthermore, the definition of~$\psi_0^*$ and \cite[Table 1]{lessard:sander:wanner:17a} yield
\begin{displaymath}
  b_\alpha(\lambda_0,0)  = \psi_0^* D_u F(\lambda_0,u_0) = 0 \; .
\end{displaymath}
The oddness of the nonlinearity~$f$ and~$\mu = 0$ further show
that~$F(\lambda,-u) = -F(\lambda,u)$ for arbitrary $\lambda \in \R$ and
$u \in X$, and this in turn implies
\begin{displaymath}
  G(\lambda, -v, -w) =
  (I-P) F(\lambda,-(v + w)) =
  -(I-P) F(\lambda,v + w) =
  -G(\lambda, v,w) \; .
\end{displaymath}
Thus, the function~$w = -W(\lambda,v) \in \tilde{X}$ solves the equation
$G(\lambda,-v,w)=0$, and the uniqueness property of~$W$ established in
Proposition~\ref{prop:lyap} then gives $-W(\lambda,v) = W(\lambda,-v)$.
One then obtains
\begin{displaymath}
  -b(\lambda,\alpha) =
  -\psi_0^* P F(\lambda,\alpha \phi_0 + W(\lambda,\alpha \phi_0)) =
  \psi_0^* P F(\lambda, -\alpha \phi_0 + W(\lambda,-\alpha \phi_0)) =
  b(\lambda,-\alpha) \; ,
\end{displaymath}
which immediately gives the trivial solution~$b(\lambda,0)=0$ for all
$\lambda \in \R$, as well as $b_{\alpha\alpha}(\lambda,0) = 0$. 

If we now again apply Proposition~\ref{prop:lyap}, then the trivial
solution of~$b$ gives rise to the smooth solution curve~$\lambda \mapsto
W(\lambda,0) \in \tilde{X}$. In fact, one can show that this solution
branch lies in~$X_a$, since in view of~$N(L|_{X_a}) = \{ 0 \}$ we can 
apply the implicit function theorem to the restriction~$F : \R \times
X_a \to Y_a$.

In order to establish the bifurcating branch which breaks the $X_a$-symmetry,
we define a function~$r$ in a neighborhood of~$(\lambda_0,0)$ by setting
\begin{displaymath}
  r(\lambda,\alpha) = \left\{ \begin{array}{rcl}
                        \DS \frac{b(\lambda,\alpha)}{\alpha} & \mbox{for} &
                          \alpha \neq 0 \; , \\[3ex]
                        \DS b_\alpha(\lambda,0)
                          & \mbox{for} & \alpha = 0 \; .
                      \end{array} \right.
\end{displaymath}
One can easily show that~$r$ is smooth. Moreover, one can show as
in~\cite[Proposition~2.11]{lessard:sander:wanner:17a} that~$r$ has the
expansion
\begin{displaymath}
  r(\lambda_0 + \nu, \alpha) =  
    \nu b_{\lambda\alpha}(\lambda_0,0) +
    \frac{\nu^2}{2} b_{\lambda\lambda\alpha}(\lambda_0,0) +
    \frac{\alpha \nu}{2} b_{\lambda\alpha\alpha}(\lambda_0,0) +
    \frac{\alpha^2}{6} b_{\alpha\alpha\alpha}(\lambda_0,0) +
    R(\nu,\alpha)
\end{displaymath}
with $R(\nu,\alpha) = O(\| (\nu,\alpha) \|^3)$. One clearly has
$r(\lambda_0, 0) = b_\alpha(\lambda_0,0) = 0$. Furthermore, it was 
shown in~\cite[Proposition~2.11]{lessard:sander:wanner:17a} that
\begin{displaymath}
  r_{\lambda}(\lambda_0,0) = b_{\lambda\alpha}(\lambda_0,0) =
  \psi_0^* D_{\lambda u}F(\lambda_0,u_0)[\phi_0] +
  \psi_0^* D_{uu}F(\lambda_0,u_0)[\phi_0,\xi_0] \neq 0 \; ,
\end{displaymath}
with~$\xi_0$ as defined uniquely in~(\ref{thm:bifexist1}). The implicit
function theorem then yields a smooth function~$\alpha \mapsto h(\alpha)$
which is defined near~$\alpha = 0$, satisfies $h(0) = \lambda_0$,
and such that in a neighborhood of~$(0,0)$ one has
\begin{displaymath}
  r(\lambda, \alpha) = 0
  \quad\mbox{ if and only if }\quad
  \lambda = h(\alpha) \; .
\end{displaymath}
This establishes the second solution branch $\alpha \mapsto
\alpha \phi_0 + W(h(\alpha), \alpha \phi_0)$. Finally, one can
follow the proof of~\cite[Proposition~2.11]{lessard:sander:wanner:17a}
verbatim to show that the two solution curves together form an
actual pitchfork bifurcation. This is accomplished by deriving an
explicit formula for~$h'(0)$ and showing that it vanishes, which 
then completes the proof of the theorem.
\end{proof}
\begin{figure}
\begin{center}
\includegraphics[width=0.95\textwidth]{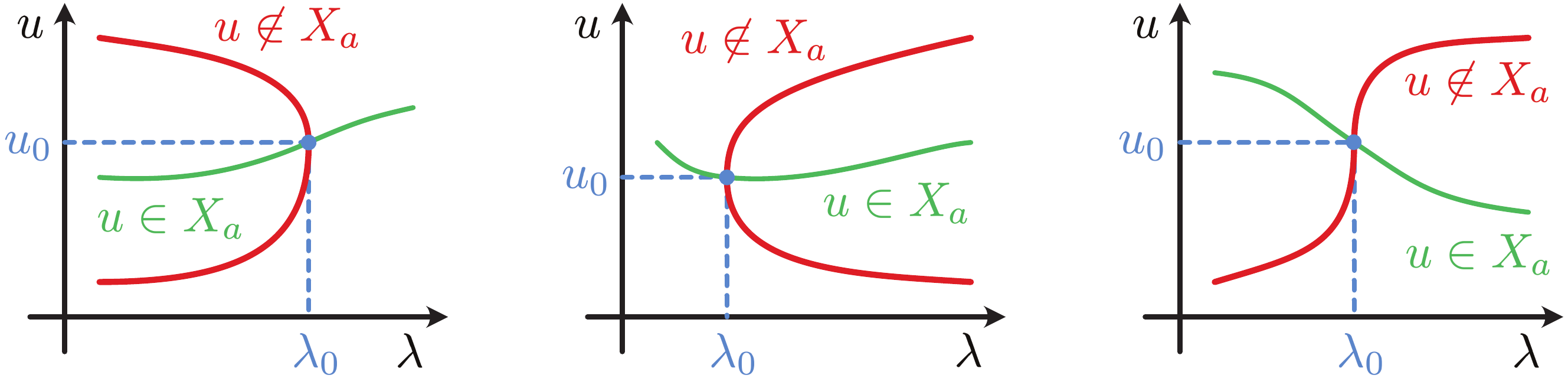}
\caption{\label{fig:symbifsketch}
         Possible symmetry-breaking pitchfork bifurcation
         scenarios in Theorem~\ref{thm:bifexist}. In the generic case,
         i.e., if the constant~$\rho$ defined in~(\ref{thm:bifexist3})
         is nonzero, one observes a classic pitchfork bifurcation ---
         and the cases $\rho > 0$ and $\rho < 0$ are shown in the first
         and second panels, respectively. However, in the case $\rho = 0$
         one could observe a situation depicted in the right-most panel.}
\end{center}
\end{figure}

\medskip
We would like to point out that in the above theorem, we classify
a bifurcation as a pitchfork bifurcation if the bifurcating solution
branch is tangential to the space~$\{ \lambda_0 \} \times X$. In 
order to actually get the parabolic shape of the bifurcating branch
that is usually associated with the pitchfork bifurcation, one needs
to verify another non-degeneracy condition. To illustrate this, one
can show that the function~$h$ constructed in the above proof
satisfies not only $h(0) = \lambda_0$ and $h'(0) = 0$, but also
$h''(0) = -\rho/3$, where the constant~$\rho$ is given by
\begin{equation} \label{thm:bifexist3}
  \rho = \frac{\psi_0^* D_{uuu}F(\lambda_0,u_0)[\phi_0,\phi_0,\phi_0]
           + 3 \psi_0^* D_{uu}F(\lambda_0,u_0)[\phi_0,\zeta_0]}
           {\psi_0^* D_{\lambda u}F(\lambda_0,u_0)[\phi_0]
           + \psi_0^* D_{uu}F(\lambda_0,u_0)[\phi_0,\xi_0]} \; ,
\end{equation}
and~$\zeta_0 \in X_a$ is defined by the equation
\begin{displaymath}
  D_{u}F(\lambda_0,u_0)[\zeta_0] +
  (I-P)D_{uu}F(\lambda_0,u_0)[\phi_0,\phi_0] = 0 \; .
\end{displaymath}
If the ratio~$\rho$ is positive, then the solutions on the parabolic
branch exist for $\lambda < \lambda_0$ close to the bifurcation point,
if~$\rho$ is negative then they exist for~$\lambda>\lambda_0$. If, on 
the other hand, one has~$\rho = 0$, either half of the branch could lie
on either side of~$\lambda_0$. These cases are illustrated in
Figure~\ref{fig:symbifsketch}. For more details, we refer the reader
to the discussion in~\cite{lessard:sander:wanner:17a}.
\subsection{Pitchfork bifurcations via extended systems}
\label{subsec:pitbif2}
With Theorem~\ref{thm:bifexist} we have established an explicit sufficient
condition for the existence of a pitchfork bifurcation in the diblock copolymer
model which is induced by the action of the $\Z_{2n}$-symmetry given by the
operator~$T_n u(x) = - u(x+1/n)$. As was pointed out in~\cite{lessard:sander:wanner:17a},
however, this condition is ill-suited if one would like to derive computer-assisted
proofs for the existence of curves of such bifurcation points in a two-parameter
setting. More useful in this situation is a reformulation of the existence
result in terms of a zero finding problem for an extended system --- and
this reformulation can be adapted to our current setting.

To explain this in more detail, consider again the space decompositions
$X = X_a \oplus X_b \oplus X_c$ and $Y = Y_a \oplus Y_b \oplus Y_c$ which
were introduced in Definitions~\ref{def:abcx} and~\ref{def:abcy}. In addition,
consider a fixed element $\ell \in X^*$ in the dual space of~$X$. We then
introduce the following extended system for~$F$, which is modeled after
the one we used in~\cite{lessard:sander:wanner:17a}:
\begin{equation} \label{eqn:cfe}
  \cF_e(\lambda,u,v) = (0,0,0)
  \quad\mbox{ for }\quad
  \cF_e : \left\{ 
    \begin{array}{c}
      \R \times X_a \times X \to \R \times Y_a \times Y \\[0.5ex]
      (\lambda,u,v) \mapsto (\ell(v)-1, F(\lambda,u), D_uF(\lambda,u)[v])
    \end{array}
  \right. .
\end{equation}
The operator~$\cF_e$ is well-defined in view of Lemma~\ref{lem:xapreserved}.
Furthermore, its Fr{\'e}chet derivative is an operator in $\cL(\R \times X_a
\times X, \; \R \times Y_a \times Y)$ which is explicitly given by
\begin{eqnarray} \label{eqn:derivcfe}
  D\cF_e(\lambda, u, v)[\tilde{\lambda},\tilde{u},\tilde{v}] & = &
    \left( \ell(\tilde{v}) \; , \;\;
    \tilde{\lambda} D_\lambda F(\lambda,u) + D_u F(\lambda,u)[\tilde{u}] \; ,
    \right. \nonumber \\[0.5ex]
  & & \;\;\;\left. \tilde{\lambda} D_{\lambda u} F(\lambda,u)[v] +
    D_{uu} F (\lambda,u)[v,\tilde{u}] + D_u F(\lambda,u)[\tilde{v}]
    \right) \; .
\end{eqnarray}
As the next main result of this section shows, the existence of a 
nondegenerate zero of this extended system is equivalent to the
sufficient condition for an $\Z_{2n}$-induced symmetry-breaking
pitchfork bifurcation given in Theorem~\ref{thm:bifexist}.
\begin{theorem}[$\Z_{2n}$-symmetry breaking pitchfork bifurcation via extended systems]
\label{thm:necsuf}
As before, consider the one-dimensional domain $\Omega = (0,1)$, the total mass
$\mu = 0$, and let the function~$f$ be a smooth and odd nonlinearity. Furthermore,
let~$F : X \to Y$ be defined as in~(\ref{def:Flambdau}) and~(\ref{def:xyspaces}).
Then the following two statements hold.
\begin{itemize}
\item[(a)] Suppose that all assumptions of Theorem~\ref{thm:bifexist} are
satisfied, and let~$\ell \in X^*$ be such that $\ell(\phi_0) = 1$. Then the
Fr{\'e}chet derivative~$D\cF_e(\lambda_0,u_0,\phi_0)$ of the mapping
in~(\ref{eqn:cfe}) is invertible, i.e., the solution~$(\lambda_0,u_0,\phi_0)
\in \R \times X_a \times X_\tau$ of the extended system
\begin{equation} \label{thm:necsuf1}
  \cF_e(\lambda,u,\phi) = (0,0,0)
\end{equation}
is an isolated non-degenerate zero.
\item[(b)] Conversely, if there exists an~$\ell \in X^*$ and
a~$\phi_0 \in X_b \cup X_c$ such that~$(\lambda_0,u_0,\phi_0)$
is a zero of the map~$\cF_e$, and if the Fr{\'e}chet
derivative~$D\cF_e(\lambda_0,u_0,\phi_0)$ is invertible, then the
nonlinear operator~$F$ satisfies all assumptions of
Theorem~\ref{thm:bifexist}.
\end{itemize}
In other words, the diblock copolymer equilibrium problem defined
earlier in~(\ref{eqn:dbcp}) undergoes a $\Z_{2n}$-symmetry breaking
pitchfork bifurcation at the point~$(\lambda_0, u_0)$ in the sense of
Theorem~\ref{thm:bifexist}, if and only if~$(\lambda_0, u_0, \phi_0)
\in \R \times X_a \times X_\tau$ is a non-degenerate zero
of~(\ref{thm:necsuf1}) for $\tau \in \{ b,c \}$. Note, however,
that for this we consider~$\cF_e$ as an operator defined
on~$\R \times X_a \times X$, even though~$\phi_0$ has to be contained
in~$X_\tau$.
\end{theorem}
\begin{proof}
We begin by establishing the validity of {\em (a)\/}. It is clear that
the assumptions of Theorem~\ref{thm:bifexist}, in combination with
$\ell(\phi_0) = 1$, imply that~$(\lambda_0,u_0,\phi_0) \in \R \times 
X_a \times X_\tau$ is a solution of~(\ref{thm:necsuf1}), where $\tau
\in \{ b,c \}$.

In order to verify that the Fr\'echet derivative~$D\cF_e(\lambda_0,u_0,
\phi_0)$ is one-to-one, suppose there exists $(\tilde{\lambda},\tilde{u},
\tilde{v}) \in \R \times X_a \times X$ such that
\begin{equation} \label{thm:necsuf2}
  D\cF_e(\lambda_0,u_0,\phi_0)[\tilde{\lambda},\tilde{u},\tilde{v}]
  = (0,0,0) \; .
\end{equation}
We show that this implies $(\tilde{\lambda},\tilde{u},\tilde{v}) = (0,0,0)$.
Assume first that the inequality $\tilde{\lambda} \ne 0$ holds. Then in view
of $\phi_0 \notin X_a$ and Proposition~\ref{prop:phi0}{\em (c)\/} we know
that $L(X_a) = Y_a$. Lemma~\ref{lem:xapreserved} and the definition of~$P$
imply $D_\lambda F(\lambda_0,u_0) \in Y_a = L(X_a) \subset R(L) \subset N(P)$,
which in turn gives the identity $D_\lambda F(\lambda_0,u_0) = (I-P) D_\lambda
F(\lambda_0,u_0)$. Now the second component of~(\ref{thm:necsuf2}), which
can be made explicit via~(\ref{eqn:derivcfe}), can be rewritten in the form
\begin{displaymath}
  (I-P)D_\lambda F(\lambda_0,u_0) +
    L[ \tilde{u} / \tilde{\lambda}] = 0 \; ,
\end{displaymath}
and since~$\xi_0 \in X_a$ is the unique solution of~(\ref{thm:bifexist1}),
this immediately furnishes $\xi_0 = \tilde{u} / \tilde{\lambda}$. The
third component of~(\ref{thm:necsuf2}) and~(\ref{eqn:derivcfe}) then
gives
\begin{displaymath}
  D_{\lambda u} F(\lambda_0,u_0)[\phi_0] +
    D_{uu} F (\lambda_0,u_0)[\phi_0,\xi_0] =
  -D_u F(\lambda_0,u_0)[\tilde{v} / \tilde{\lambda}] =
  -L[\tilde{v} / \tilde{\lambda}] \in R(L) \; ,
\end{displaymath}
which contradicts our assumption~(\ref{thm:bifexist2}), and we therefore
obtain $\tilde{\lambda} = 0$. But then the second component of
both~(\ref{thm:necsuf2}) and~(\ref{eqn:derivcfe}) implies $L[\tilde{u}] = 0$,
and thus also $\tilde{u} \in N(L) \cap X_a = \{ 0 \}$, i.e., we necessarily
have $\tilde{u} = 0$. Substituting both $\tilde{\lambda} = 0$ and $\tilde{u} = 0$
into the third component finally gives $L \tilde{v} = 0$, as well as $\tilde{v} =
\alpha \phi_0$. Together with the first component of~(\ref{thm:necsuf2})
this further yields $0 = \ell(\alpha \phi_0) = \alpha \ell(\phi_0) = \alpha$,
i.e., one has the identity $\tilde{v}  = 0$. This completes the proof
that~$D\cF_e(\lambda_0,u_0,\phi_0)$ is one-to-one.

We now show that~$D\cF_e(\lambda_0,u_0,\phi_0)$ is onto. For this, let
$(\tau,y,z) \in \R \times Y_a \times Y$ be arbitrary, but fixed. We need
to explicitly construct an inverse image under the Fr\'echet derivative.
To this end, notice first that since $y \in Y_a = L(X_a)$ and
$N(L) \cap X_a = \{ 0 \}$, there exists a unique element $\tilde{u} \in X_a$
such that $L[\tilde{u}] = y$. In view of Hypothesis~\ref{hyp:1dk}, the linear
functional $\psi_0^* \in Y^*$ satisfies $R(L) = N(\psi_0^*)$. Now define
\begin{displaymath}
  \tilde{\lambda} =
  \frac{\psi_0^* \left( z - D_{uu}F(\lambda_0,u_0)[\phi_0,\tilde{u}]\right)}
    {\psi_0^* \left( D_{\lambda u}F(\lambda_0,u_0)[ \phi_0 ] +
     D_{uu}F(\lambda_0,u_0)[\phi_0,\xi_0] \right)} \; ,
\end{displaymath}
where the denominator of this ratio is nonzero due to~(\ref{thm:bifexist2}).
A simple algebraic reformulation of this definition then leads to
\begin{displaymath}
  \psi_0^* \left( \tilde{\lambda} D_{\lambda u}F(\lambda_0,u_0)[ \phi_0 ] +
    D_{uu}F(\lambda_0,u_0)[\phi_0,\tilde{u} + \tilde{\lambda} \xi_0] -
    z \right) = 0 \; .
\end{displaymath}
The choice of~$\psi_0^*$ shows that the argument in the above
equation has to be contained in~$R(L)$, and since $L[\beta \phi_0] = 0$
for any scalar~$\beta$, there exists a $\tilde{v} \in X$ such that for
every $\beta \in \R$ the equation
\begin{equation} \label{eqn:comp2}
  \tilde{\lambda} D_{\lambda u}F(\lambda_0,u_0)[ \phi_0 ] +
    D_{uu}F(\lambda_0,u_0)[\phi_0,\tilde{u} + \tilde{\lambda}\xi_0] +
    L[\tilde{v} + \beta \phi_0] = z
\end{equation}
is satisfied. Now~(\ref{thm:bifexist1}), in combination with
$N(P) = R(L)$ and $D_\lambda F(\lambda_0,u_0) \in Y_a \subset R(L)$
give
\begin{equation} \label{eqn:comp3}
  \tilde{\lambda} D_\lambda F (\lambda_0,u_0) +
    L[\tilde{u} + \tilde{\lambda}\xi_0] =
  \tilde{\lambda} \left( (I-P)  D_\lambda F (\lambda_0,u_0) +
    L [\xi_0] \right) + L[\tilde{u}] =
  L[\tilde{u}] = y \; . 
\end{equation}
Finally, notice that $\xi_0 \in X_a$ yields $\tilde{u} + \tilde{\lambda}
\xi_0 \in X_a$, and this in turn implies that for all $\beta \in \R$
the identities in~(\ref{eqn:comp2}) and~(\ref{eqn:comp3}) establish
the second and third components of the desired equation
\begin{displaymath}
  D\cF_e(\lambda_0,u_0,\phi_0)[\tilde{\lambda} , \;
    \tilde{u} + \tilde{\lambda} \xi_0, \;
    \tilde{v} + \beta \phi_0] = (\tau,y,z) \; .
\end{displaymath}
It remains to choose~$\beta$ in such a way that the first component of
the equation holds as well. Since $\ell(\phi_0) = 1$ by our earlier
normalization, we need to solve $\tau = \ell(\tilde{v} + \beta \phi_0)
= \ell(\tilde{v}) + \beta$, which is clearly satisfied in we let
$\beta = \tau - \ell(\tilde{v})$. This shows that~$D\cF_e(\lambda_0,
u_0, \phi_0)$ is onto, and completes the proof of~{\em (a)\/}.

We now turn our attention to the verification of~{\em (b)\/}, i.e.,
we assume that there exists an~$\ell \in X^*$ and a~$\phi_0 \in X_b
\cup X_c$ such that~$(\lambda_0,u_0,\phi_0)$ is a zero of the
map~$\cF_e$, and that the Fr{\'e}chet derivative~$D\cF_e(\lambda_0,
u_0, \phi_0) \in \cL(\R \times X_a \times X, \; \R \times Y_a \times Y)$
is invertible. We need to show that all the assumptions of
Theorem~\ref{thm:bifexist} are satisfied.

We begin by establishing Hypotheses~\ref{hyp:fred} and~\ref{hyp:1dk}.
In view of $\cF_e(\lambda_0,u_0,\phi_0) = 0$ and~(\ref{eqn:cfe}),
one immediately obtains $F(\lambda_0,u_0) = 0$, as well as
$L[\phi_0] = D_uF(\lambda_0,u_0)[\phi_0] = 0$. Furthermore, due
to~$\ell(\phi_0) = 1$ we have $\dim N(L) \ge 1$. On the other hand,
since the Fr{\'e}chet derivative~$D\cF_e(\lambda_0, u_0, \phi_0)$
is invertible, it can be shown as in~\cite[Proof of
Theorem~2.12]{lessard:sander:wanner:17a} that we have in fact
$\dim N(L) = 1$. Since in this paper it was also shown that~$L$
is a Fredholm operator with index zero, this establishes both
Hypotheses~\ref{hyp:fred} and~\ref{hyp:1dk}. In particular,
it follows that~$P$ and~$Q$ as defined in Definition~\ref{def:orthproj}
have rank one. Notice also that all the assumptions of
Proposition~\ref{prop:phi0} have been established, and~{\em (c)\/}
of this result, in combination with our assumption $\phi_0 \in X_b
\cup X_c$, then immediately implies $L(X_a) = Y_a$.

In order to establish the remaining assumptions of Theorem~\ref{thm:bifexist},
we first show that the equation in~(\ref{thm:bifexist1}) has a unique 
solution~$\xi_0 \in X_a$. We know that $u_0 \in X_a$, and this yields
the inclusion $D_\lambda F(\lambda_0,u_0) \in Y_a \subset R(L)$.
Definition~\ref{def:orthproj} implies $N(P) = R(L)$, and therefore
\begin{displaymath}
  (I-P) D_\lambda F(\lambda_0,u_0) =
  D_\lambda F(\lambda_0,u_0) \in Y_a
\end{displaymath}
holds. This implies the existence of a $\xi_0 \in X_a$ with
$L \xi_0 = -(I-P) D_\lambda F(\lambda_0,u_0)$ in view of $L(X_a) = Y_a$,
i.e., equation~(\ref{thm:bifexist1}) holds. Moreover, this
solution~$\xi_0$ is uniquely determined, because if $\hat{\xi}_0 \in X_a$
were another solution, then $L\hat{\xi}_0 = -(I-P) D_\lambda F(\lambda_0,u_0)
= L\xi_0$, i.e., one has $L(\hat{\xi}_0 - \xi_0) = 0$. But this equality
yields the inclusion $\hat{\xi}_0 - \xi_0 \in N(L) \cap X_a = \{ 0 \}$. 

To conclude our proof, we only have to establish~(\ref{thm:bifexist2}).
For this, let~$z$ be any element of the complement $Y \setminus R(L)$.
Due to the assumptions of~{\em (b)\/} there exists a triple~$(\tilde{\lambda},
\tilde{u},\tilde{v}) \in \R \times X_a \times X$ such that the identity
$D\cF_e (\lambda_0, u_0, \phi_0)[\tilde{\lambda},\tilde{u},\tilde{v}] =
(0,0,z)$ is satisfied. Let~$\xi_0$ again denote the unique solution
to~(\ref{thm:bifexist1}) from above. We assume first that
$\tilde{\lambda} = 0$. Then the explicit form of the second component
given in~(\ref{eqn:derivcfe}) yields $0 = \tilde{\lambda} D_\lambda
F(\lambda_0,u_0) + L \tilde{u} = L \tilde{u}$, which in turn gives
the inclusion $\tilde{u} \in N(L) \cap X_a = \{ 0 \}$, and thus $\tilde{u} = 0$. 
Therefore, another application of~(\ref{eqn:derivcfe}) furnishes
\begin{displaymath}
  z = \tilde{\lambda} D_{\lambda u} F(\lambda_0,u_0)[\phi_0] +
    D_{uu} F(\lambda_0,u_0)[\phi_0,\tilde{u}] + L \tilde{v} =
  L \tilde{v} \; ,
\end{displaymath}
which contradicts $z \notin R(L)$. Thus, our assumption
concerning~$\tilde{\lambda}$ was wrong, and we have to have the
inequality $\tilde{\lambda} \ne 0$. But then~(\ref{eqn:derivcfe})
gives rise to $0 = \tilde{\lambda} D_\lambda F(\lambda_0,u_0) +
L \tilde{u}$, and after division by~$\tilde{\lambda}$ one obtains
\begin{equation} \label{eqn:dlaml}
  D_\lambda F(\lambda_0,u_0) + L[ \tilde{u} / \tilde{\lambda}] = 0 \; . 
\end{equation}
Notice that we already established earlier the inclusion $Y_a =
L(X_a) \subset R(L) = N(P)$, as well as $D_\lambda F(\lambda_0,u_0)
\in Y_a \subset N(P)$, and this immediately gives $P D_\lambda
F(\lambda_0,u_0) = 0$. In combination with~(\ref{eqn:dlaml}) one
then obtains
\begin{displaymath}
  (I-P) D_\lambda F(\lambda_0,u_0) +
    L[ \tilde{u} / \tilde{\lambda}] = 0 \; ,
\end{displaymath}
and since~$\xi_0$ is the unique solution to this latter equation,
we have to have $\xi_0 = \tilde{u} / \tilde{\lambda}$. Substituting
this into the third component of $D\cF_e (\lambda_0, u_0, \phi_0)
[\tilde{\lambda},\tilde{u},\tilde{v}] = (0,0,z)$, one finally obtains
after a few algebraic reformulations
\begin{displaymath}
  D_{\lambda u} F(\lambda_0,u_0)[\phi_0] +
    D_{uu} F(\lambda_0,u_0)[\phi_0,\xi_0] =
  z / \tilde{\lambda} - L \tilde{v} \; .
\end{displaymath}
Clearly we have $L \tilde{v} \in R(L)$. Thus, if the right-hand side of
this equation were contained in~$R(L)$, then one would have to have
$z \in R(L)$, which contradicts our original assumption. Therefore, the
right-hand side cannot be an element of~$R(L)$, and this
establishes~(\ref{thm:bifexist2}). This completes the proof of
the theorem.
\end{proof}

\medskip
The above result is remarkable in the sense that even though we
are considering a completely different symmetry from the simple ones
discussed in~\cite{lessard:sander:wanner:17a}, we still obtain 
essentially the same sufficient existence condition for a symmetry-breaking
pitchfork bifurcation via the extended system~(\ref{eqn:cfe}). All that
changes is the restriction of the second argument~$u$ to the new
symmetry space. In fact, a closer inspection of our results shows that
the same approach should work in other situations as well, as long as
a few basic assumptions are satisfied. These are collected in the 
following remark.
\begin{remark}[Required assumptions for the extended system approach]
\label{remark:metathm}
One can easily verify that our main Theorems~\ref{thm:bifexist}
and~\ref{thm:necsuf} remain valid, as long as the following eight
assumptions are satisfied:
\begin{itemize}
\item The underlying Hilbert spaces allow for decompositions of the
form $X = X_a \oplus X_b \oplus X_c$ and $Y = Y_a \oplus Y_b \oplus Y_c$,
where the involved spaces are pairwise orthogonal, see
Lemma~\ref{lem:pairwiseorthogonal}.
\item For every $u_0 \in X_a$ one obtains $D_\lambda F(\lambda_0, u_0) \in Y_a$,
as shown in Lemma~\ref{lem:xapreserved}.
\item For $u_0 \in X_a$ and $\tau = a,b,c$ the inclusions
$L X_\tau \subset Y_\tau$ are satisfied, see Lemma~\ref{lem:subsets}.
\item For $u_0 \in X_a$ and $\tau = b,c$ one has $D_{\lambda u}
F(\lambda_0,u_0) [X_\tau] \subset Y_\tau$, see Lemma~\ref{lem:subsets}.
\item For $u_0 \in X_a$ and $\tau = b,c$ we have $D_{uu}F(\lambda_0,u_0)
[X_\tau,X_a] \subset Y_\tau$, as in Lemma~\ref{lem:derivcontain}.
\item The orthogonal projector from Definition~\ref{def:orthproj} satisfies
$P(Y_a) \subset Y_a$, see Lemma~\ref{lem:subsets}.
\item The underlying nonlinear operator~$F$ is odd, i.e., we have
$F(\lambda,-u) = -F(\lambda,u)$.
\item If the kernel function satisfies $\phi_0 \notin X_a$, then $L(X_a) = Y_a$,
see Lemma~\ref{prop:phi0}.
\end{itemize} 
\end{remark}
%
%
%
%
\section{Validation of Symmetry-Induced Pitchfork Bifurcations}
\label{sec:compval}
In this section we combine the theory developed in the earlier parts of the
paper with a constructive version of the implicit function theorem to establish
the existence of branches of pitchfork bifurcation points which are induced 
through the cyclic group action defined in~(\ref{eqn:z2nsymm}). For this,
we recall a specific branch-validation version of the the constructive implicit
function theorem from~\cite{sander:wanner:16a} in Section~\ref{sec:compval1}, and
also describe in detail the system that has to be studied in this context. After
that, Section~\ref{sec:compval2} demonstrates how the assumptions of the branch
validation result can be verified. After briefly outlining our spectral approach
to this, we show how our recent paper~\cite{rizzi:etal:22a} can be used to 
determine the necessary norm bound of the inverse of the Fr\'echet derivative,
and we also derive required Lipschitz estimates. Finally, Section~\ref{sec:compval3}
presents some sample pitchfork curve continuations.
\subsection{Establishing branches of pitchfork bifurcation points}
\label{sec:compval1}
In view of Theorem~\ref{thm:necsuf} we can establish the existence of a
specific pitchfork bifurcation point by proving that the associated extended
system~(\ref{eqn:cfe}) has an isolated zero. In our situation, this extended
system involves three unknowns --- the equilibrium solution~$u$, the
kernel function~$v$, and the parameter value~$\lambda$. Note, however,
that the diblock copolymer model has an additional parameter~$\sigma$,
and it was shown in~\cite{johnson:etal:13a, lessard:sander:wanner:17a}
that these bifurcation points combine to form curves parameterized
by~$\sigma$. In the present section, we will explain how a constructive
version of the implicit function theorem can be used to rigorously
verify these branches in the setting of cyclic symmetries. 

For the purposes of this paper, we are interested in finding stationary
solutions of the diblock copolymer model which are in fact pitchfork
bifurcation points. As equilibrium solutions, they have to be zeros of
the nonlinear operator
\begin{equation} \label{eqn:compval1}
  F(\sigma,\lambda,u) =
  -\Delta ( \Delta u + \lambda f(u + \mu) ) - \lambda \sigma u
  \; ,
\end{equation}
which is considered as an operator $F : \R \times \R \times X \to Y$ for
the spaces defined in~(\ref{def:xyspaces}), and where in contrast to our 
earlier usage we also explicitly indicate its dependence on~$\sigma$. Due
to Theorem~\ref{thm:necsuf}, such a zero is a pitchfork bifurcation point
at the parameter values~$\sigma$ if it is an isolated zero of the extended
operator
\begin{equation} \label{eqn:compval2}
  \cF(\sigma,\cdot,\cdot,\cdot) : \left\{ 
    \begin{array}{c}
      \R \times X_a \times X \to \R \times Y_a \times Y \\[0.5ex]
      (\lambda,u,v) \mapsto (\ell(v)-1, F(\sigma,\lambda,u),
       D_uF(\sigma,\lambda,u)[v])
    \end{array}
  \right. ,
\end{equation}
where the spaces~$X_a$ and~$Y_a$ were defined in Definitions~\ref{def:abcx}
and~\ref{def:abcy}, respectively. Notice that we include the explicit
dependence on the parameter~$\sigma$, which is the natural continuation
parameter for curves of bifurcation points. In order to simplify the 
notation going forward, we make use of the abbreviations
\begin{equation} \label{eqn:compval3}
  w = (\lambda,u,v) \in \cX = \R \times X_a \times X
  \qquad\mbox{ and }\qquad
  \cY = \R \times Y_a \times Y \; .
\end{equation}
In addition, for applying the generalization of the constructive implicit
function theorem from~\cite{sander:wanner:16a} which is taylor-made for branch
validation, one needs to verify a number of assumptions through rigorous
computer-assisted means. More precisely, one has to accomplish the following:
\begin{itemize}
\item[(H1)] Find an  approximate zero~$w^* = (\lambda^*,u^*,v^*) \in \cX$ of
the operator~$\cF(\sigma^*,\cdot) : \cX \to \cY$ such that for a real
number~$\rho > 0$ one has the residual estimate $\|\cF(\sigma^*,w^*)\|_{\cY}
\le \rho$. This can be done by simply using interval arithmetic, based on a
truncated cosine Fourier series representation of the functions~$u^* \in X_a$
and~$v^* \in X$, see also Remark~\ref{rem:fourierabc} and the discussion
in the next section.
\item[(H2)] Find a bound $K \ge 0$ such that $\| D_w \cF(\sigma^*,w^*)^{-1}
\|_{\cL(\cY,\cX)} \le K$. This estimate is by far the technically most involved
one, but due to the specific form of~$\cF$, we can directly quote a result
from~\cite{rizzi:etal:22a}. This will also be presented in the next section,
and the reader can find all the technical details in the cited paper.
\item[(H3)] Find Lipschitz bounds for the partial Fr\'echet derivatives~$D_w\cF$
and~$D_{\sigma}\cF$ of the extended operator~$\cF$ for all~$(\sigma,w)$
close to~$(\sigma^*,w^*)$ in the following sense. There exist four Lipschitz
constants~$M_k \ge 0$ for $k = 1,\ldots,4$, as well as~$d_w > 0$
and~$d_\sigma > 0$, such that for all pairs~$(\sigma,w) \in \R \times \cX$
with $\| w - w^* \|_\cX \le d_w$ and $|\sigma - \sigma^*| \le d_\sigma$
one has
\begin{eqnarray*}
  \left\| D_w\cF(\sigma,w) -
    D_w\cF(\sigma^*,w^*) \right\|_{\cL(\cX,\cY)} & \le &
    M_1 \left\| w - w^* \right\|_\cX +
    M_2 \left|\sigma - \sigma^* \right| \; , \\[1.5ex]
  \left\| D_\sigma\cF(\sigma,u) -
    D_\sigma\cF(\sigma^*,w^*) \right\|_\cY & \le &
    M_3 \left\| w - w^* \right\|_\cX +
    M_4 \left|\sigma - \sigma^* \right| \; ,
\end{eqnarray*}
where~$\| \cdot \|_{\cL(\cX,\cY)}$ denotes the operator norm
in~$\cL(\cX,\cY)$, and as usual we identify~$\cY$ with~$\cL(\R,\cY)$.
These estimates are substantially more straightforward than the previous
step, and they will be established in the next section.
\end{itemize}
Under these assumptions, one can then establish the following 
theorem which guarantees branch segments of zeros of~$\cF$ parameterized
by~$\sigma$ close to the approximate solution~$(\sigma^*,w^*)$. This
result is taken from~\cite[Theorem~5]{sander:wanner:16a}, and it is a
simple consequence of the constructive implicit function theorem
in~\cite[Theorem~1]{sander:wanner:16a}. In fact, the theorem below
reduces to the original constructive implicit function theorem in
the case~$w^\circledast = 0$.
\begin{theorem}[Regular branch segment validation]
\label{branch:thm}
Let~$\cX$ and~$\cY$ be Banach spaces, and suppose that the nonlinear
parameter-dependent operator $\cF : \R \times \cX \to \cY$ is both
Fr\'echet differentiable and satisfies~(H3). Assume that~$(\sigma^*,w^*)
\in \R \times \cX$ satisfies the estimates~(H1) and~(H2) for some positive
constants~$\rho$ and~$K$, and let~$w^\circledast \in \cX$
be given with
\begin{equation} \label{branch:thm1}
  \left\| D_\sigma\cF(\sigma^*,w^*) +
    D_w\cF(\sigma^*,w^*)[w^\circledast] \right\|_\cY \le \eta
\end{equation}
for some constant~$\eta \ge 0$, which will indicate the slant of
the box containing the solution branch. Finally, assume that we 
have the estimates
\begin{equation} \label{branch:thm2}
  4 K^2 \rho M_1 < 1
  \qquad\mbox{ and }\qquad
  2 K \rho < d_w \; .
\end{equation}
Then there exist pairs of constants~$(\delta_\sigma,\delta_w)$
which satisfy
\begin{equation} \label{branch:thm3}
  0 < \delta_\sigma \le d_\sigma \; , \qquad
  0 < \delta_w \le d_w \; , 
  \qquad\mbox{ and }\qquad
  \delta_\sigma \left\| w^\circledast \right\|_\cX + \delta_w \le d_w \; ,
\end{equation}
as well as the two inequalities
\begin{equation} \label{branch:thm4}
  2 K M_1 \delta_w + 2 K
    \left( M_1 \left\| w^\circledast \right\|_\cX + M_2 \right)
    \delta_\sigma \le 1
\end{equation}
and
\begin{equation} \label{branch:thm5}
  2 K \rho + 2 K \eta \delta_\sigma + 2 K
    \left( M_1 \left\| w^\circledast \right\|_\cX^2
           + (M_2+M_3) \left\| w^\circledast \right\|_\cX 
           + M_4 \right)
  \delta_\sigma^2
    \le \delta_w  \; ,
\end{equation}
and for each such pair the following holds. For every parameter~$\sigma
\in \R$ with $|\sigma - \sigma^*| \le \delta_\sigma$ there exists a
unique~$w(\sigma) \in \cX$ with $\| w(\sigma) - ( w^* + (\sigma -
\sigma^*) w^\circledast)\|_\cX \le \delta_w$, and for which the
nonlinear equation $\cF(\sigma, w(\sigma)) = 0$ holds. In other words,
all solutions of the nonlinear problem $\cF(\sigma,w)=0$ in the
slanted set
\begin{displaymath}
  \left\{ (\sigma,w) \in \R \times \cX \; : \;
  \left| \sigma - \sigma^* \right| \le \delta_\sigma
  \quad\mbox{and}\quad
  \left\| w - \left( w^* +
    \left(\sigma - \sigma^* \right) w^\circledast \right) \right\|_\cX
  \le \delta_w \right\}
\end{displaymath}
lie on the branch $\sigma \mapsto w(\sigma)$. In addition,
if the mapping~$\cF : \R \times \cX \to \cY$ is $k$-times
continuously Fr\'echet differentiable, then so is the solution
function $\sigma \mapsto w(\sigma)$.
\end{theorem}
For an illustration of the above theorem we refer the reader
to~\cite[Figure~7]{wanner:18a}. Note that for the application of this result,
and of course also for the verification for the assumptions~(H2) and~(H3),
it is crucial to have the partial Fr\'echet derivatives of~$\cF$ at hand.
One can easily show that they are given by
\begin{eqnarray}
  D_w\cF(\sigma,w)[\tilde{w}] & = &
    \left( \ell(\tilde{v}) \; , \;\;
    \tilde{\lambda} D_\lambda F(\sigma,\lambda,u) +
    D_u F(\sigma,\lambda,u)[\tilde{u}] \; ,
    \right. \nonumber \\[0.5ex]
  & & \;\;\;\left. \tilde{\lambda} D_{\lambda u}
    F(\sigma,\lambda,u)[v] +
    D_{uu} F (\sigma,\lambda,u)[v,\tilde{u}] +
    D_u F(\sigma,\lambda,u)[\tilde{v}]
    \right) \label{eqn:compval4a1} \\[0.5ex]
   & = &
    \left( \ell(\tilde{v}) \; , \;\;
    \tilde{\lambda} (-\Delta f(u+\mu) - \sigma u) -
    \Delta ( \Delta \tilde{u} + \lambda f'(u + \mu) \tilde{u} )
    - \lambda \sigma \tilde{u} \; ,
    \right. \nonumber \\[0.5ex]
  & & \quad\;\;\;\left. \tilde{\lambda} (-\Delta f'(u+\mu)v - \sigma v) -
    \Delta( \lambda f''(u + \mu) v \tilde{u})
    \right. \nonumber \\[0.5ex]
  & & \quad\;\;\;\left. \qquad -
    \Delta ( \Delta \tilde{v} + \lambda f'(u + \mu) \tilde{v} )
    - \lambda \sigma \tilde{v} \right) \; ,
    \label{eqn:compval4a2}
\end{eqnarray}
where we write $w = (\lambda,u,v)$ and $\tilde{w} =
(\tilde{\lambda},\tilde{u},\tilde{v})$, as well as
\begin{equation} \label{eqn:compval4b}
  D_\sigma\cF(\sigma,w) \; = \;
  \left( 0, \; D_\sigma F(\sigma,\lambda,u), \;
    D_{\sigma u} F(\sigma,\lambda,u)[v] \right) \; = \;
  \left( 0, \; -\lambda u, \; -\lambda v \right) \; .
\end{equation}
\subsection{Verifying the assumptions for the computer-assisted proofs}
\label{sec:compval2}
We now address the verification of assumptions~(H1) through~(H3) of
Theorem~\ref{branch:thm}. With the exception of the last of these, all
of them can be treated as in our previous papers~\cite{rizzi:etal:22a,
sander:wanner:21a, wanner:18b}. In view of this, we only provide a short
descriptions and leave the details to the cited references.

We begin our discussion by illustrating how~(H1) can be established. It
was shown in Remark~\ref{rem:fourierabc} that the crucial spaces~$X_a$,
$X_b$, and~$X_c$ have straightforward explicit Fourier cosine series
representations. Thus, it is natural to find the pitchfork bifurcation
point approximation in the form of a truncated series. If we denote the
resulting discretization size by~$N \in \N$, then we consider the orthogonal
projection~$P_N : X \to X$ defined via
\begin{equation} \label{def:projectionPN}
  P_N u(x) \; := \; \sum_{k=1}^N a_k \cos(k \pi x)
  \quad\mbox{ for every }\quad
  u(x) \; = \; \sum_{k=1}^\infty a_k \cos(k \pi x)
  \;\;\mbox{ in~$X$} \; ,
\end{equation}
see also~(\ref{rem:fourierabc1}). An analogous projection~$Q_N$ can also
be defined on the image space~$Y$. Thus, one can project both the second
and the third component of the extended nonlinear operator~$\cF$
in~(\ref{eqn:compval2}) onto the spaces~$Q_N Y_a$ and~$Q_N Y$, respectively,
and by only allowing arguments~$u^* \in P_N X_a$ and $v^* \in P_N X$ one then
obtains a finite-dimensional system which can be solved numerically for
the solution approximation~$u^*$ and the kernel function~$v^*$, at the
approximate parameter values~$\lambda^*$. Notice that the dimension
of this system is given by~$1 + \lfloor N/n \rfloor + N$, as long as~$N$
is larger than~$n$. By choosing appropriate Hilbert space norms on
the spaces~$X$ and~$Y$ as in~\cite{rizzi:etal:22a, sander:wanner:21a},
one can then easily compute an upper bound~$\rho$ on the residual based
on the Fourier cosine sum representations of~$u^*$ and~$v^*$. In fact,
for computational convenience we use the norms~$\| u \|_X = \| \Delta u
\|_{L^2(0,1)}$ and~$\| u \|_Y = \| \Delta^{-1} u \|_{L^2(0,1)}$, which
are equivalent to the respective standard Sobolev norms on these spaces.
Moreover, the rigorous upper bound is established using interval arithmetic,
more precisely, the Matlab package INTLAB~\cite{rump:99a}.
\begin{table}
\begin{center}
  \begin{tabular}{|l|c|c||l|c|c|}
  \hline
  & $\cL$ & $D_w\cF(\sigma^*,w^*)$ & & $\cL$ & $D_w\cF(\sigma^*,w^*)$ \\ \hline
  Spaces & $U_1$ & $X_a$ & Spaces & $V_1$ & $Y_a$ \\
  & $U_2$ & $X$ & & $V_2$ & $Y$ \\ \hline
  Arguments & $\eta_1$ & $\tilde{\lambda}$ & Coefficients & $\alpha_{11}$ & $0$ \\
  & $v_1$ & $\tilde{u}$ & & $\ell_{11}$ & $0$ \\
  & $v_2$ & $\tilde{v}$ & & $\ell_{12}$ & $\ell$ \\ \hline
  Coefficients & $\beta_1$ & $1$ & Coefficients & $\beta_2$ & $1$ \\
  & $b_{11}$ & $\Delta f(u^*+\mu) + \sigma^* u^*$ & & $b_{21}$ &
    $\Delta f'(u^*+\mu)v^* + \sigma^* v^*$ \\
  & $c_{11}$ & $\lambda^* f'(u^* + \mu)$ & & $c_{21}$ & $\lambda^* f''(u^* + \mu) v^*$ \\
  & $c_{12}$ & $0$ & & $c_{22}$ & $\lambda^* f'(u^* + \mu)$ \\
  & $\gamma_{11}$ & $\lambda^* \sigma^*$ & & $\gamma_{21}$ & $0$ \\
  & $\gamma_{12}$ & $0$ & & $\gamma_{22}$ & $\lambda^* \sigma^*$ \\ \hline
  \end{tabular}
  \vspace*{0.3cm}
  \caption{\label{table:thm:k}
           Reformulating the Fr\'echet derivative~$D_w\cF(\sigma^*,w^*)$ as the
           linear elliptic operator~$\cL$ defined in equations~(\ref{thm:k0}),
           (\ref{thm:k1}), and ~(\ref{thm:k2}). In our situation, we have
           $p = 1$ and $q = 2$, and the spaces, arguments, and coefficients in
           the respective operator definitions correspond to each other as
           outlined in the table.}
\end{center}
\end{table}

We now turn our attention to the hypothesis~(H2). The required inverse norm
bound for the Fr\'echet derivative~$D_w\cF(\sigma^*,w^*)$ presented
in~(\ref{eqn:compval4a2}) can be established directly using the results
of~\cite{rizzi:etal:22a}. In this paper, we developed a method based on
the Neumann series and the construction of a suitable approximate
inverse to compute a rigorous bound on the inverse operator norm of certain
fourth-order linear elliptic operators which include scalar constraints.
More precisely, in~\cite[Theorem~4.1]{rizzi:etal:22a} we considered
a linear operator
\begin{equation} \label{thm:k0}
  \cL : \R^p \times \prod_{i=1}^q U_i \to
    \R^p \times \prod_{i=1}^q V_i \; ,
\end{equation}
where $U_i \subset X$ and $V_i \subset Y$ are suitably chosen closed
subspaces, which acts on the argument vector~$(\eta_1, \ldots, \eta_p,
v_1, \ldots, v_q)$, and whose first~$p$ components are given by
\begin{equation} \label{thm:k1}
  \sum_{i=1}^p \alpha_{ki} \eta_i +
    \sum_{j=1}^q \ell_{kj}(v_j)
  \qquad\mbox{ for }\qquad
  k = 1,\ldots,p \; ,
\end{equation}
while the remaining~$q$ functional components are
\begin{equation} \label{thm:k2}
  -\beta_{k} \Delta^2 v_k 
    - \sum_{i=1}^p b_{ki} \eta_i
    - \Delta \sum_{j=1}^q c_{kj} v_j
    - \sum_{j=1}^{q} \gamma_{kj} v_j
  \qquad\mbox{ for }\qquad
  k = 1,\ldots,q \; .
\end{equation}
Clearly, the operator~$D_w\cF(\sigma^*,w^*)$ defined in~(\ref{eqn:compval4a2})
falls into this category, and the necessary correspondences are collected
in Table~\ref{table:thm:k}. Thus, we can simply apply this theorem to 
compute the norm estimate, and we refer the readers
to~\cite[Theorem~4.1]{rizzi:etal:22a} for more details.

As the final step, we have to establish the Lipschitz estimates required
in~(H3). This can be accomplished similar to our proceeding in~\cite{rizzi:etal:22a,
sander:wanner:21a}, so we will keep our discussion as short as possible. For this,
we define for every $\ell \in \N_0$ the constant
\begin{equation} \label{def:lipschitz:01}
  f^{(\ell)}_{\max} \; := \;
  \max_{|\rho| \le \|u^*\|_\infty + \overline{C}_1 d_w}
    \left| f^{(\ell)}(\rho + \mu) \right| \; ,
  \qquad\mbox{ where }\qquad
  \overline{C}_1 = 0.149072
\end{equation}
denotes the embedding constant from Sobolev's embedding theorem in one space dimension
introduced in~\cite[Lemma~2.3]{sander:wanner:21a}, see also~\cite{wanner:18a}. Furthermore,
consider pairs~$(\sigma,w)$ and~$(\sigma^*,w^*)$ in~$\R \times \cX$, where~$\cX$ was
defined in~(\ref{eqn:compval3}), which satisfy both $|\sigma - \sigma^*| \le d_\sigma$
and $\| w - w^* \| \le d_w$. Then the definition of the operator~$F$ in~(\ref{eqn:compval1})
implies the estimates
\begin{eqnarray}
  & & \hspace*{-2cm}
    \| D_\lambda F(\sigma,\lambda,u) - D_\lambda F(\sigma^*,\lambda^*,u^*) \|_Y
    \nonumber \\
  & \le & \| \Delta f(u+\mu) + \sigma u - \Delta f(u^*+\mu) - \sigma^* u^* \|_Y
    \nonumber \\
  & \le & \| f(u+\mu) - f(u^*+\mu)\|_{L^2}
    + \| \sigma u - \sigma u^* \|_Y + \| \sigma u^* - \sigma^* u^* \|_Y
    \nonumber \\
  & \le & \frac{\pi^2 f^{(1)}_{\max} + |\sigma^*| + d_\sigma}{\pi^4} \,
    \| u - u^* \|_X + \| u^* \|_Y \, \left| \sigma - \sigma^* \right|
    \label{def:lipschitz:02} \; ,
\end{eqnarray}
where we also used the estimates $\| u \|_{L^2} \le \| u \|_X / \pi^2$
and $\| u \|_Y \le \| u \|_X / \pi^4$ for all $u \in X$, see for
example~\cite[Lemma~2.6]{sander:wanner:21a}. Similarly, one obtains
for all $\tilde{u} \in X$ the estimate
\begin{eqnarray*}
  & & \hspace*{-2cm}
    \| D_u F(\sigma,\lambda,u)[\tilde{u}] -
       D_u F(\sigma^*,\lambda^*,u^*)[\tilde{u}] \|_Y \\
  & \le & \| \Delta ( \lambda f'(u + \mu) \tilde{u} -
    \lambda^* f'(u^* + \mu) \tilde{u} ) \|_Y +
    | \lambda \sigma - \lambda^* \sigma^* | \, \| \tilde{u} \|_Y \\
  & \le & |\lambda| \, \| f'(u + \mu) \tilde{u} - f'(u^* + \mu) \tilde{u} \|_{L^2} +
    | \lambda - \lambda^* | \, \| f'(u^* + \mu) \tilde{u} \|_{L^2} \\
  & & \qquad +
    | \lambda - \lambda^* | \, \frac{|\sigma|}{\pi^4} \, \| \tilde{u} \|_X +
    | \sigma - \sigma^* | \, \frac{|\lambda^*|}{\pi^4} \, \| \tilde{u} \|_X \\
  & \le & |\lambda| \, f^{(2)}_{\max} \, \| u - u^* \|_\infty \| \tilde{u} \|_{L^2} +
    | \lambda - \lambda^* | \, \| f'(u^* + \mu) \|_\infty \, \| \tilde{u} \|_{L^2} \\
  & & \qquad +
    | \lambda - \lambda^* | \, \frac{|\sigma|}{\pi^4} \, \| \tilde{u} \|_X +
    | \sigma - \sigma^* | \, \frac{|\lambda^*|}{\pi^4} \, \| \tilde{u} \|_X \; ,
\end{eqnarray*}
which in turn implies
\begin{eqnarray}
  & & \hspace*{-2cm}
    \| D_u F(\sigma,\lambda,u) -
       D_u F(\sigma^*,\lambda^*,u^*) \|_{\cL(X,Y)} \nonumber \\
  & \le & \frac{\pi^2 \| f'(u^* + \mu) \|_\infty + |\sigma^*| + d_\sigma}{\pi^4}
    \, | \lambda - \lambda^* | \nonumber \\
  & & \qquad +
    \frac{\overline{C}_1 f^{(2)}_{\max} (|\lambda^*|+d_w)}{\pi^2}
    \, \| u - u^* \|_X +
      \frac{|\lambda^*|}{\pi^4} \,
    | \sigma - \sigma^* | \; .
    \label{def:lipschitz:03}
\end{eqnarray}
We now start estimating the two terms which remain in the last component
of the operator~$\cF$. On the one hand, we have
\begin{eqnarray}
  & & \hspace*{-1.5cm}
    \| D_{\lambda u}F(\sigma,\lambda,u)[v] -
    D_{\lambda u}F(\sigma^*,\lambda^*,u^*)[v^*]\|_Y \nonumber \\
  & \le & \| \Delta(f'(u+\mu)v - f'(u^*+\mu)v^*) \|_Y +
    \| \sigma v - \sigma^* v^* \|_Y \nonumber \\
  & \le & \|f'(u+\mu)v -f'(u^*+\mu)v^* \|_{L^2} +
    |\sigma| \, \| v - v^* \|_Y + \| v^* \|_Y \, |\sigma - \sigma^*|
    \nonumber \\
  & \le & \|f'(u+\mu)v - f'(u+\mu) v^*\|_{L^2} +
    \|f'(u+\mu)v^* - f'(u^*+\mu)v^* \|_{L^2} \nonumber \\
  & & \qquad +
    \frac{|\sigma|}{\pi^4} \, \| v - v^* \|_X + \| v^* \|_Y \,
    |\sigma - \sigma^*| \nonumber \\
  & \le & \frac{f^{(1)}_{\max}}{\pi^2} \, \|v-v^*\|_X +
    \frac{f^{(2)}_{\max} \, \|v^*\|_\infty}{\pi^2} \, \| u-u^* \|_{X} +
    \frac{|\sigma|}{\pi^4} \, \| v - v^* \|_X + \| v^* \|_Y \,
    |\sigma - \sigma^*| \nonumber \\
  & \le & 
    \frac{f^{(2)}_{\max} \, \|v^*\|_\infty}{\pi^2} \, \| u-u^* \|_{X} +
    \frac{\pi^2 f^{(1)}_{\max} + |\sigma^*| + d_\sigma}{\pi^4} \,
    \|v-v^*\|_X + \| v^* \|_Y \, |\sigma - \sigma^*| \; ,
    \label{def:lipschitz:04}
\end{eqnarray}
while on the other hand one obtains the estimate
\begin{eqnarray*}
  & & \hspace*{-1.5cm}
    \| D_{uu}F(\sigma,\lambda,u)[v,\tilde{u}] -
    D_{uu}F(\sigma^*,\lambda^*,u^*)[v^*,\tilde{u}] \|_Y \\
  & \le & 
    \| \Delta( \lambda f''(u + \mu) v \tilde{u} -
    \lambda^* f''(u^* + \mu) v^* \tilde{u}) \|_Y \\
  & \le & \| \lambda f''(u + \mu) v \tilde{u} -
    \lambda f''(u^* + \mu) v \tilde{u} \|_{L^2} +
    \| \lambda f''(u^* + \mu) v \tilde{u} -
    \lambda^* f''(u^* + \mu) v^* \tilde{u} \|_{L^2} \\
  & \le & |\lambda| \, f^{(3)}_{\max} \, \| u - u^* \|_\infty
    \, \| v \|_\infty \, \| \tilde{u} \|_{L^2} +
    \| \lambda f''(u^* + \mu) v \tilde{u} -
    \lambda f''(u^* + \mu) v^* \tilde{u} \|_{L^2} \\
  & & \qquad +
    \| \lambda f''(u^* + \mu) v^* \tilde{u} -
    \lambda^* f''(u^* + \mu) v^* \tilde{u} \|_{L^2} \\
  & \le & |\lambda| \, f^{(3)}_{\max} \, \| u - u^* \|_\infty
    \, \| v \|_\infty \, \| \tilde{u} \|_{L^2} +
    |\lambda| \, \| f''(u^* + \mu) \|_\infty \,
    \| v - v^* \|_\infty \, \| \tilde{u} \|_{L^2} \\
  & & \qquad + |\lambda - \lambda^*| \,
    \| f''(u^* + \mu) v^* \|_\infty \, \| \tilde{u} \|_{L^2}
    \; ,
\end{eqnarray*}
which in turn implies
\begin{eqnarray}
  & & \hspace*{-1.5cm}
    \| D_{uu}F(\sigma,\lambda,u)[v,\cdot] -
    D_{uu}F(\sigma^*,\lambda^*,u^*)[v^*,\cdot] \|_{\cL(X,Y)}
    \nonumber \\
  & \le & \frac{\| f''(u^* + \mu) v^* \|_\infty}{\pi^2} \, 
    |\lambda - \lambda^*| +
    \frac{\overline{C}_1 f^{(3)}_{\max} (|\lambda^*| + d_w)
    (\| v^* \|_\infty + \overline{C}_1 d_w)}{\pi^2} \, \| u - u^* \|_X
    \nonumber \\
  & & \qquad +
    \frac{\overline{C}_1 \| f''(u^* + \mu) \|_\infty (|\lambda^*| + d_w)}{\pi^2} \,
    \| v - v^* \|_X \; . \label{def:lipschitz:05}
\end{eqnarray}
Altogether, we have established the estimates
\begin{eqnarray*}
  \| D_\lambda F(\sigma,\lambda,u) - D_\lambda
    F(\sigma^*,\lambda^*,u^*) \| & \le &
    c_1 \| u - u^* \| +
    c_2 |\sigma - \sigma^*| \; , \\
  \| D_u F(\sigma,\lambda,u) -
    D_u F(\sigma^*,\lambda^*,u^*) \| & \le &
    c_3 | \lambda - \lambda^* | +
    c_4 \| u - u^* \| +
    c_5 |\sigma - \sigma^*| \; , \\
  \| D_{\lambda u}F(\sigma,\lambda,u)[v] -
    D_{\lambda u}F(\sigma^*,\lambda^*,u^*)[v^*]\| & \le &
    c_6 \| u-u^* \| +
    c_7 \|v-v^*\| + 
    c_8 |\sigma - \sigma^*| \; , \\
  \| D_{uu}F(\sigma,\lambda,u)[v,\cdot] -
    D_{uu}F(\sigma^*,\lambda^*,u^*)[v^*,\cdot] \| & \le &
    c_9 |\lambda - \lambda^*| +
    c_{10} \| u - u^* \| +
    c_{11} \| v - v^* \| \; ,
\end{eqnarray*}
where the Lipschitz constants~$c_k$ can be inferred from
equations~(\ref{def:lipschitz:01}) through~(\ref{def:lipschitz:05}),
and we dropped the subscripts indicating the specific norms.
After these preparations, one can now easily establish the estimates
in~(H3). For this, we write the nonlinear operator~$\cF$ in component
form as $\cF = (\cF_1,\cF_2,\cF_3)$. Then we clearly have
\begin{displaymath}
  D_w\cF_1(\sigma,w)[\tilde{w}] -
    D_w\cF_1(\sigma^*,w^*)[\tilde{w}] = 0 \; ,
\end{displaymath}
while for the second component one obtains
\begin{eqnarray*}
  & & \hspace*{-1.25cm}
    \| D_w\cF_2(\sigma,w)[\tilde{w}] -
    D_w\cF_2(\sigma^*,w^*)[\tilde{w}] \|_Y \\
  & \le & \| D_\lambda F(\sigma,\lambda,u) -
    D_\lambda F(\sigma^*,\lambda^*,u^*) \|_Y |\tilde{\lambda}| +
    \| D_u F(\sigma,\lambda,u)[\tilde{u}] -
    D_u F(\sigma^*,\lambda^*,u^*)[\tilde{u}] \|_Y \\
  & \le & \left( c_1 \| u - u^* \|_X + c_2 |\sigma - \sigma^*|
    \right) |\tilde{\lambda}|
    + \left( c_3 | \lambda - \lambda^* | +
    c_4 \| u - u^* \|_X + c_5 |\sigma - \sigma^*| \right)
    \| \tilde{u} \|_X \\
  & \le & \left( c_3 | \lambda - \lambda^* | +
    (c_1+c_4) \| u - u^* \|_X + (c_2+c_5) |\sigma - \sigma^*| \right)
    \| \tilde{w} \|_{\cX} \\
  & \le & \left( \sqrt{c_3^2 + (c_1+c_4)^2} \,
    \| w - w^* \|_{\cX} + (c_2+c_5) |\sigma - \sigma^*| \right)
    \, \| \tilde{w} \|_{\cX} \; ,
\end{eqnarray*}
and similarly for the third component
\begin{eqnarray*}
  & & \hspace*{-1.0cm}
    \| D_w\cF_3(\sigma,w)[\tilde{w}] -
    D_w\cF_3(\sigma^*,w^*)[\tilde{w}] \|_Y \\
  & \le & \left( c_6 \| u-u^* \|_X + c_7 \|v-v^*\|_X + 
    c_8 |\sigma - \sigma^*| \right) |\tilde{\lambda}| \\
  & & \qquad +
    \left( c_9 |\lambda - \lambda^*| + c_{10} \| u - u^* \|_X +
    c_{11} \| v - v^* \|_X \right) \| \tilde{u} \|_X \\
  & & \qquad +
    \left( c_3 | \lambda - \lambda^* | +
    c_4 \| u - u^* \|_X +
    c_5 |\sigma - \sigma^*| \right) \| \tilde{v} \|_X \\
  & \le & \left( (c_3+c_9) | \lambda - \lambda^* | +
    (c_4+c_6+c_{10}) \| u - u^* \|_X \right. \\
  & & \qquad \left. + 
    (c_7+c_{11}) \| v - v^* \|_X + 
    (c_5+c_8) |\sigma - \sigma^*| \right)
    \| \tilde{w} \|_{\cX} \\
  & \le & \left( \sqrt{(c_3+c_9)^2 + (c_4+c_6+c_{10})^2
    + (c_7+c_{11})^2} \, \| w - w^* \|_{\cX} + 
    (c_5+c_8) |\sigma - \sigma^*| \right)
    \| \tilde{w} \|_{\cX} \; .
\end{eqnarray*}
If we now define the constants~$M_1$ and~$M_2$ as
\begin{eqnarray}
  M_1 & = & \sqrt{2 \max\left\{ c_3^2 + (c_1+c_4)^2 \, , \;
    (c_3+c_9)^2 + (c_4+c_6+c_{10})^2 + (c_7+c_{11})^2
    \right\}} \; , \label{def:lipschitz:06} \\
  M_2 & = & \sqrt{2} \, \max\left\{ c_2+c_5 \, , \;
    c_5+c_8 \right\} \; , \label{def:lipschitz:07}
\end{eqnarray}
then one immediately obtains
\begin{displaymath}
  \| D_w\cF(\sigma,w)[\tilde{w}] -
    D_w\cF(\sigma^*,w^*)[\tilde{w}] \|_{\cY} \; \le \;
  \left( M_1 \| w - w^* \|_{\cX} + M_2 |\sigma - \sigma^*|
    \right) \| \tilde{w} \|_{\cX} \; ,
\end{displaymath}
i.e., the first estimate in~(H3) holds. Furthermore, in view
of~(\ref{eqn:compval4b}) we have the estimate
\begin{eqnarray*}
  & & \hspace*{-1.5cm}
    \| D_\sigma\cF(\sigma,w) - D_\sigma\cF(\sigma^*,w^*)
    \|_{\cY} \\
  & \le & \left( \| \lambda u - \lambda^* u^* \|_Y^2 +
    \| \lambda v - \lambda^* v^* \|_Y^2 \right)^{1/2} \\
  & \le & \sqrt{2} \, \left( \| \lambda u - \lambda u^* \|_Y^2 +
    \| \lambda u^* - \lambda^* u^* \|_Y^2 +
    \| \lambda v - \lambda v^* \|_Y^2 +
    \| \lambda v^* - \lambda^* v^* \|_Y^2 \right)^{1/2} \\
  & \le & \sqrt{2} \, \left( |\lambda|^2 \, \| u - u^* \|_Y^2 +
    \| u^* \|_Y^2 \, |\lambda - \lambda^*|^2 +
    |\lambda|^2 \, \| v - v^* \|_Y^2 +
    \| v^* \|_Y^2 \, |\lambda - \lambda^*|^2 \right)^{1/2} \\
  & \le & \sqrt{2} \, \max\left\{ \frac{|\lambda^*| + d_w}{\pi^4}
    \, , \; \sqrt{\| u^* \|_Y^2 + \| v^* \|_Y^2} \right\} \,
    \| w - w^* \|_{\cX} \; ,
\end{eqnarray*}
and therefore the second estimate in~(H3) is satisfied with
\begin{equation} \label{def:lipschitz:08}
  M_3 \; = \;
    \sqrt{2} \, \max\left\{ \frac{|\lambda^*| + d_w}{\pi^4}
    \, , \; \sqrt{\| u^* \|_Y^2 + \| v^* \|_Y^2} \right\}
  \quad\mbox{ and }\quad
  M_4 \; = \; 0 \; .
\end{equation}
This completes the verification of the assumptions of the regular
branch segment validation result in Theorem~\ref{branch:thm}.
\subsection{Sample computational validations}
\label{sec:compval3}

\begin{table}
  \begin{center}
  \begin{tabular}{|c|c|c|c|c|c|c|} \hline
$n$ &$\lambda$ 	 &  		$N$	   &    $\tau$   &    $K$     &  $M_1$    &        $d_w$      \\ \hline
  5 &  115.69    &          178    &    0.49917  &     73.453 &    54.859 &        1.2408e-04 \\   
  7 &  315.57    &          670    &    0.29987  &    154.78  &   150.85  &        2.1415e-05 \\
  4 &  336.05    &          874    &    0.49972  &    503.86  &   253.14  &        3.9201e-06 \\
  6 &  769.12    &         2000    &    0.29979  &    829.58  &   636.9   &        9.4642e-07   
     \\ \hline
  \end{tabular}
  \end{center}
  \caption{\label{table:valid}
    Validation parameters for the validated solutions.}
\end{table}

In this section, we implement the techniques listed above in order to computationally validate
the first two solutions shown in Figure~\ref{fig:oddodd}, and the first two solutions shown
in Figure~\ref{fig:eveneven} for the fixed parameter value $\sigma = 6$. The computed
validation parameters are given in Table~\ref{table:valid}. For larger~$\lambda$
values, the computationally necessary value of~$N$ becomes extremely large if one uses the
model equations in their original unmodified form. Notice that this is unavoidable, as the
number of modes required to represent the solutions increases quickly with increasing~$\lambda$.
In addition, in this limit the equation becomes closer to singular, and therefore one fully
expects that numerical approaches become significantly more difficult.

Nevertheless, our focus in this paper is on the analytical background for a new type of
symmetry-breaking pitchfork bifurcation. For the sake of space, we therefore do not address
the computationally heavy methods needed to refine this method, and choose to address this
numerical machinery in a future work. The following are two specific techniques which we
plan to incorporate in future. 

In order to improve the validation and to address the validation of further solutions with
larger~$n$ values and substantially larger~$\lambda$ values at bifurcation, we would need
to include preconditioning in our validation, which essentially amounts to rescalings in the
underlying partial differential equation. In particular, since both~$\lambda$ versus~$(u,v)$,
and the components of function~$\cF_e$, occur on extremely different length scales, one would
expect to see rather substantial stiffness in the computation of both the Lipschitz constants
and the bound~$K$. See for example the improvement in the method due to preconditioning
in the recent paper~\cite{kamimoto:kim:sander:wanner:22a}. After preconditioning is established,
we would further be able to consider the case of varying~$\sigma$, and create a validated
continuation method to find the curve of bifurcation points in the two-parameter family.
In~\cite{kamimoto:kim:sander:wanner:22a}, we developed a validated pseudo-arclength
continuation method. In future work, we intend to adapt this method to the current setting. 

\section*{Acknowledgments}
The research of E.S.\ and T.W.\ was partially supported by the
Simons Foundation under Awards~636383 and~581334, respectively.

%
%

\addcontentsline{toc}{section}{References}
\footnotesize
%
%
\bibliography{wanner1a,wanner1b,wanner2a,wanner2b,wanner2c}
\bibliographystyle{abbrv}
\end{document}